\RequirePackage{fix-cm}
\documentclass[smallextended]{svjour3}       % onecolumn (second format)
\smartqed  % flush right qed marks, e.g. at end of proof
\usepackage{graphicx,color}
\usepackage{graphics}%,lineno}
\usepackage{amssymb,amsmath}

\def \x {{\bf x}}

\def \V {{\bf V}}
\def \O \boldsymbol{\Omega}

\def\E {{\bf E}}
\def\S {\mathcal{S}^2}
\def\ve {\varepsilon}

\usepackage[colorlinks]{hyperref}
\hypersetup{
	linkcolor=blue,
}
\begin{document}

\title{Effective Rheological Properties in Semidilute Bacterial Suspensions}%\thanks{Grants or other notes


%\subtitle{Effective Viscosity in Bacterial Suspensions}


\author{   
        Mykhailo Potomkin       \and
        Shawn D. Ryan %etc.
	\and
	Leonid Berlyand
}

\institute{
M. Potomkin \at
{Department of Mathematics,
The Pennsylvania State University,
University Park, PA 16802}\\
\email{mup20@math.psu.edu} 
\and 
S.~D. Ryan \at
Department of Mathematical Sciences and Liquid Crystal Institute,
Kent State University,
Kent, OH 44240\\
\email{sryan18@kent.edu} 
\and
L. Berlyand \at
Department of Mathematics,
The Pennsylvania State University,
University Park, PA 16802\\
\email{berlyand@math.psu.edu}           %  \\
}

\date{Received: date / Accepted: date}

\maketitle

\begin{abstract}
	Interactions between swimming bacteria have led to remarkable experimentally observable macroscopic properties such as the reduction of the effective viscosity, enhanced mixing, and diffusion.  In this work, we study an individual based model for a suspension of interacting point dipoles representing bacteria in order to gain greater insight into the physical mechanisms responsible for the drastic reduction in the effective viscosity. In particular, asymptotic analysis is carried out on the corresponding kinetic equation governing the distribution of bacteria orientations.   This  allows one to derive an explicit asymptotic formula for the effective viscosity of the bacterial suspension in the limit of bacterium non-sphericity. The results show good qualitative agreement with numerical simulations and previous experimental observations. Finally, we justify our approach by proving existence, uniqueness, and regularity properties for this kinetic PDE model.

\end{abstract}
 \setcounter{tocdepth}{5}
 \tableofcontents

 	   \section{\bf Introduction}
 Bacterial suspensions exhibit remarkable macroscopic properties due to the emergence of self-organization among its components.   In particular, interesting effective properties  such as enhanced diffusivity, the formation of sustained whorls and jets, and the ability to extract useful work among other results  have been recently observed for suspensions of bacteria, such as {\it Bacillus subtilis} \cite{Wu,Sokolov2,LepGuaGolPesGol09,Sokolov3,CisKesGanGol11}.   The striking experimental observations on the effective viscosity provide the motivation for studying a suspension's effective properties;  namely, the observation of a seven-fold reduction in the effective viscosity of a suspension of swimming {\it B. subtilis} \cite{Sokolov}.  This reduction is observed below $2\%$ volume fraction typically referred to as the {\it dilute regime} where bacteria are far apart and essentially interact with the background fluid only.   With the assumption of no interbacterial interactions, this regime has been studied analytically  in recent works (e.g., \cite{Saintillan1,Haines1,Haines2,Haines3}). There bacterial tumbling was introduced in order for the formula to predict a decrease in the effective viscosity \cite{Haines3}.  However, in the absence of tumbling (e.g., for anaerobic bacteria) the decrease is still observed experimentally \cite{Sokolov}. 
 It was shown recently in \cite{RyaHaiBerZieAra11} that interbacterial interactions substantially contribute to effective viscosity and an estimate for this contribution was given. 
 Rigorous analysis of this contribution and its corresponding effect on the effective viscosity of the suspension is the main component of this paper. 
 %In our work, we seek to use rigorous analysis      
 
 %The role of interbacterial interactions in the reduction of effective viscosity was first investigated in \cite{RyaHaiBerZieAra11}.
 
 %Therefore, the role of interbacterial interactions in the reduction of effective viscosity must be investigated.  
 
 %an explicit formula for the effective viscosity  needs to be developed that takes into account interbacterial interactions, yet still predicts a decrease even when bacteria do not tumble.  
 %In the regime of viscosity decrease the stress due to bacterial activity depends significantly on the set of bacteria orientations.  
 
We begin with an individual based model (IBM) previously introduced in  \cite{RyaHaiBerZieAra11,RyaBerHaiKar12}, which has been successfully used to capture the decrease in the effective viscosity and other collective phenomena.  Such suspensions, where interbacterial interactions play an important role and are modeled as a sum of pairwise interactions, are referred to as {\it semi-dilute}.  Our goal is to identify the underlying mechanisms that contribute to the decrease of the effective viscosity in this concentration regime.  
% derive an explicit asymptotic formula for the effective viscosity in this concentration regime.  
The main tool we employ is a kinetic theory derived from this individual based model.

 The purpose for employing a kinetic approach is to replace a large system of coupled differential equations by a single  continuum partial differential equation with respect to a probability distribution of bacteria positions and orientations. Note that it is natural to consider probabilistic quantities since the main focus of this work is the study of the effective properties. The main computational advantage of the kinetic approach is that the number of bacteria $N$ does not increase the complexity of the problem \cite{Spoh91,BerJabPot13}.  Namely, the PDE could be solved numerically with a fixed spatial or temporal grid independent of $N$.
 %, which will make the computational cost more manageable.  
 In addition to the ability to consider many different initial conditions at once, another advantage to introducing this probabilistic framework is to consider the limiting regime as $N \to \infty$, the so-called {\it mean field limit}.   More information on kinetic equations can be found in the seminal works of the 1970's \cite{NeuWic74,BraHep77,Dob79} or more contemporary reviews \cite{Car10,JabPer00,Per04,Deg04}. %Since the main focus of this work is the study of the effective properties that represent an averaged value, it is natural to use probabilistic quantities.    
 
 %The kinetic approach has many advantages, yet the 
 Significant difficulty in the analysis come from the incorporation of interactions.  First, they appear in the kinetic equation as a non-local term due to the fact that the suspension of interacting bacteria is generally described analytically by %the probability density of 
 configurations of \emph{all} bacteria.   % In addition to considering a suspension of $N$ baxteria, 
 Second, the main interactions that are taken into account are hydrodynamic, which diverge as bacteria approach one another as the square of their distance. This results in a singular kernel in this non-local term.   Thus, the kinetic equation consists of a nonlocal, nonlinear PDE due to the presence of interactions.
 %The kinetic theory will include singular terms arising due to the presence of hydrodynamic interactions.  These terms will diverge as each bacterium approaches one another.   
 % Due to the difficulties cited above, a rigorous formula for the effective viscosity incorporating interactions in a semi-dilute suspension has yet to be derived until now. 

Using a kinetic approach, the main result of this paper is an explicit asymptotic formula for the effective viscosity %in the limit of vanishing bacterial non-sphericity. 
 with interbacterial interactions taken into account.  
 The formula reveals the physical mechanisms necessary for the decrease in effective viscosity observed experimentally.   To achieve this result we first find the steady state solution of the kinetic equation and then use this solution to compute the effective viscosity.  For completeness, we also establish the well-posedness of the kinetic equation.
%Here the non-sphericity or deviation of a bacterium's shape from that of a sphere can be measured via the Bretherton constant.  In addition, we prove the kinetic PDE used throughout is well-posed and prove the regularity of its solutions.   

  This paper is organized as follows. Section~\ref{model} begins by introducing the individual based %first principle 
  model under consideration for a semi-dilute bacterial suspension.  From this, the kinetic equation for the orientation distribution is formally derived.   The reason we begin with the IBM is that the effective properties of a suspension are derived from knowledge of microscopic configurations, which is transferred from the IBM to the kinetic model. In Section~\ref{assumption} we introduce the main conditions under which we derive the asymptotic formula for the effective viscosity and discuss their physical significance. Section \ref{derivodens} contains the derivation of the asymptotic steady state solution to the kinetic equation for the orientation distribution in the limit of small non-sphericity.  The effective viscosity from the asymptotic formula is then compared to the same quantity computed from direct simulations of the individual based model in Section \ref{evform}. The important physical mechanisms for the decrease in viscosity are identified and the orientation distribution is compared to the results of previous works in the dilute case. In addition, the normal stress differences and relaxation time are considered. The existence, uniqueness and regularity properties of a solution to the kinetic PDE are proven in Section \ref{kin_solvable}. Finally, we formulate our conclusions and outline potential future investigations in Section~\ref{conc}. 

 \section{\bf Model for Semidilute Bacterial Suspensions}\label{model}
 
 We begin by introducing the coupled PDE/ODE system governing the fluid and bacteria dynamics respectively.  Each bacterium is represented as a point force dipole.   One force represents the bacterium's propulsion mechanism (e.g., flagellar motion) and the other is the opposing viscous drag exerted by the bacterium's body on the fluid.   This approximation has been experimentally verified by observing the flow due to a bacterium (e.g., {\it Bacillus subtilis}) in a fluid and comparing it to that of a force dipole \cite{DreDunCisGanGol11}.     
 
 %In this paper, we are primarily concerned with studying bacteria known as ``pushers", which are swimming bacteria propelling themselves from behind.   This type of bacteria exhibited the dramatic viscosity reduction observed in recent experiments \cite{Sokolov}.  One can also consider ``pullers'', which propel themselves from in front of their bodies by  pulling the fluid toward themselves.  The distinction in classification is determined by the particle's propulsion mechanism.  Now that a point dipolar model for a bacterium is chosen, then one must consider how each bacterium moves through the fluid.  Throughout this work, we consider the point dipoles immersed in a planar shear flow where they begin to swim and interact with one another.  We are interested in deriving an explicit asymptotic formula for the semidilute regime, which is characterized by  the fact that the contribution to the equations of motion due to hydrodynamic interactions can be written as a sum of pairwise interactions represented by the velocity ${\bf u}^j$ in \eqref{rij-0}. 
 
 As a bacterium swims through the fluid its trajectory may be altered through interactions with other bacteria and the background flow.  At every moment in time a bacterium propels itself in the direction in which it is oriented.  If one bacterium comes into close contact with another, then a collision can occur altering the bacterium's position.  This is modeled by an excluded volume potential.  Finally, the flow itself has an impact on a bacterium trajectory through the ambient background flow and the sum of flows induced from the propulsion of all the other bacteria on its position.  To make these ideas more concrete we now introduce an individual based model (IBM), which governs a bacterium's position and orientation.
 
 We consider $N$ bacteria with the position of the center of mass of the $i$th bacterium ${\bf x}^i = (x^i,y^i,z^i)$ and orientation ${\bf d}^i = (d_1^i,d_2^i,d_3^i)$.  A bacterium's translational velocity is derived from a balance of forces due to self-propulsion, collisions, and the flow field acting on the position of the bacterium.  A bacterium's orientation velocity is derived from a balance of torques in the form of Jeffery's equation for an ellipsoid in a linear flow with additional terms due to the flows generated by the other bacteria in the suspension \cite{Jef22}.   Thus, the equations of motion for bacterial positions ${\bf x}$ and orientations ${\bf d}$ originally introduced from first principles in \cite{RyaHaiBerZieAra11} are
 \begin{eqnarray}
 \dot{\bf x}^i & =& V_0{\bf d}^i + \sum_{j \neq i}  \left( {\bf u}^j({\bf x}^i,{\bf d}^j)  +  {\bf F}^j({\bf x}^i)\right) + {\bf u}^{\text{BG}}({\bf x}^i),  \label{rij-0} \displaystyle\\
 \dot{\bf d}^i & =& -\frac{1}{2}{\bf d}^i \times \left(\boldsymbol{\omega}^{\text{BG}}_0({\bf x}^i) + \sum_{j \neq i} \boldsymbol{\omega}^j({\bf x}^i,{\bf d}^j)\right)\nonumber\\&&\hspace{15 pt}-{\bf d}^i \times  \left[B{\bf d}^i \times \left({\bf E}^{\text{BG}}_0({\bf x}^i) +  \sum_{j \neq i} {\bf E}^j({\bf x}^i,{\bf d}^j)\right) {\bf d}^i\right] + \sqrt{2D}\dot{W}, \displaystyle \label{dij}
 \end{eqnarray}
 %\frac{1}{6 \pi \eta_0 l},$l$ is the length of a particle,, and $\eta_0$ is the ambient (background fluid) viscosity.
 where $V_0$ is an individual bacterium's swimming speed and $B$ is the Bretherton constant which takes into account the geometry of the bacterium's body ($B\ll 1$: near spherical, $B \approx 1$: needle-like).   
 %For the main bacterium of interest, {\itshape B. subtilis} used in the striking experimental results of \cite{Sokolov}, $B \approx .9$.  
 The externally-imposed planar shear flow contributes to each bacterium's motion through the fluid velocity, ${\bf u}^{BG} = (0, \gamma x, 0)^T$,
 as well as its effect on a bacterium's orientation through the vorticity $\boldsymbol{\omega}^{\text{BG}}_0 = \nabla_{\x} \times {\bf u}^{BG}$ and rate of strain ${\bf E}^{\text{BG}}_0 = \frac{1}{2}\left(\nabla_{\x} {\bf u}^{BG} + (\nabla_{\x} {\bf u}^{BG})^T\right)$.    Here $W$ is a white noise and we let $D \sim B^2$ be the diffusion coefficient.  This order of $D$ will be used  throughout this work and represents the idea that the random motion present in the system has a greater effect the more elongated a particle is.%For the remainder of this work we consider contribution of some small random motion (e.g., a non-perfect swimming apparatus).
 
 The additional terms in Jeffrey's equation \eqref{dij} beyond the contribution from the background flow are due to the vorticity $ \boldsymbol{\omega}^j$ and rate of strain ${\bf E}^j$ generated by the $j$-th dipole on position of the $i$-th dipole
 \begin{align*}
 \boldsymbol{\omega}^j &= \nabla_{\x} \times {\bf u}^j, \qquad 
 {\bf E}^j = \frac{1}{2}\left(\nabla_{\x} {\bf u}^j + (\nabla_{\x} {\bf u}^j)^T\right).
 \end{align*}
 Each of these terms depends on the fluid velocity ${\bf u}^{j}$, which is governed by Stokes equation and will be described in greater detail below.  %The dependence of these terms on the orientation ${\bf d}^j \in \mathcal{S}^2$ of the $j$-th dipole is suppressed in \eqref{rij-0}-\eqref{dij} for convenience in notation. 
 
 \begin{remark}
 	The equations of motion \eqref{rij-0}-\eqref{dij} are a $5N$ coupled system of ordinary differential equations in comparison to the dilute case studied in \cite{Haines3} where there were only two ODEs governing the evolution of a single bacterium in an infinite medium (only depending on a single bacterium's orientation).  Thus, the semi-dilute system of equations adds a greater complexity than the dilute case previously studied.
 \end{remark}
 
 The use of Stokes equation to model the fluid is justified by estimating the Reynold's number.  Based on the typical size $\ell_0 \sim 1\;\mu \text{m}$  and swimming speed $V_0 \sim 20 \;\mu \text{m}/\text{s}$ of a bacterium, 
 in addition to the typical dynamic viscosity $\eta_0 \sim \;10^{-3}\; \text{Pa}\cdot \text{s}$ and density $\rho \sim \;10^{3}\; \text{kg}/\text{m}^3$ of the suspending fluid, the flow has a Reynolds number $Re$ around  $2\times 10^{-5} \ll 1$.   Thus, inertial effects can be neglected. %due to the small mass of each particle.
 Also, it is assumed that a steady-state flow is established on a timescale much smaller than  the characteristic timescale, which is the time for a bacterium to swim its length.  %This justifies the use of Stokes equation to model the fluid. 
 
 The flow at the position of bacterium $i$ due to bacterium $j$ is given by 
 %The fluid velocity acting on particle $i$ due to particle $j$ is given as
 ${\bf u}^j({\bf x}^i,{\bf d}^j) = {\bf u}({\bf x}^j-{\bf x}^i,{\bf d}^j)$ where ${\bf u}({\bf x},{\bf d})$ is a solution of the Stokes problem
 \begin{equation}\label{stokeseqnsingle}
 \left\{
 \begin{array}{lr}
 \displaystyle
 \eta_0\Delta_{\x}\mathbf{u}(\mathbf{x},{\bf d}) - \nabla_{\x} p(\mathbf{x},{\bf d}) = \nabla_{\x} \cdot \bigl[ {\bf D}(\mathbf{d}) \delta(\mathbf{x})\bigr], & \mathbf{x} \in \mathbb{R}^3,\\
 \nabla_{\x} \cdot \mathbf{u}(\mathbf{x},{\bf d}) = 0, & \mathbf{x} \in \mathbb{R}^3,\\
 \mathbf{u}({\bf x},{\bf d}) \to 0, & |{\bf x}| \to \infty,\\
 \end{array}
 \right.
 \end{equation}
 where $\eta_0$ is the ambient fluid viscosity and $p$ is the pressure.
 %, and the domain $\mbox{\mancube} := \{{\bf x} \in \mathbb{R}^3 ~|~ x_i < \frac{L}{2}\}$.
 %The fluid in a suspension of $N$ particles would be modeled the same way.  Since Stokes equation is linear the fluid velocity is just the superposition of solutions to \eqref{stokeseqnsingle}.  
 The dipole tensor ${\bf D}=\{D_{lm}\}$ is given by
 \begin{equation}\label{dipolemoment}
 D_{lm}(\mathbf{d}) := U_0\left(d_ld_m - \frac{1}{3}\delta_{lm}\right),
 \end{equation}
 where $U_0$ is %the stresslet force quantifying 
 the strength of the dipole referred to as the {\it dipole moment}. For pushers, bacteria that propel themselves from behind such as {\it B. subtilis}, $U_0<0$.  Equation \eqref{stokeseqnsingle} has an explicit solution:
 \begin{equation}
 u_k({\bf x}, {\bf d}) := \frac{1}{8\pi\eta_0}\sum_{l=1}^3\sum_{m=1}^3 D_{lm}({\bf d})\mathcal{G}_{kl,m}(\x),
 \end{equation}
 where $\mathcal{G}_{kl}({\bf x}) = \frac{1}{8\pi\eta_0}\left(\frac{\delta_{kl}}{|{\bf x}|}+\frac{{x}_k{x}_l}{|{\bf x}|^3}\right)$ is the Oseen tensor. 
 
 \begin{remark}
 	In order to study the role of interactions in semi-dilute suspensions it is natural to deal with a point representation of swimmers such that the whole suspension is modeled by points interacting in the fluid. 
 	In our paper, a swimmer is represented by a point force dipole with the dipole tensor \eqref{dipolemoment}.  
 	In general, for a given model of a swimmer, such a point representation can by found as the second order term in  the multipole expansion, see  \cite{KimKar91}.     
 	{\it We note that all results of this paper such as the asymptotic formula for orientation distribution and effective viscosity can be easily modified to semi-dilute suspensions with swimmers whose dipole tensor is different from \eqref{dipolemoment}.} 
 \end{remark}

 In order to analyze the system \eqref{rij-0}-\eqref{dij}, the associated kinetic theory for the probability density of bacterial configurations (positions and orientations of each bacterium) is studied.  In general, to derive the corresponding kinetic equation one 
 %considers the equations of motion \eqref{rij-0}-\eqref{dij} with 
 assumes that initial conditions are random.  Then each sum in the equations of motion is a sum of identically distributed random variables.  The key step in the formal derivation of the kinetic equation is replacing all sums in the equations of motion by their expectations %. using an analogous idea to the Law of Large Numbers  designed for sums of identically distributed random variables first studied in 
 \cite{Poz00,Spoh91,Jab14}.
 %\begin{equation}\label{llneqn}
 %\sum_{j \neq i} A({\bf r}^{(i,j)},{\bf d}^j) \to \frac{1}{|V_L|} \int_{\mathcal{S}^2} \int_{V_L}  A ({\bf x} - {\bf x}^\prime, {\bf d}^\prime) P({\bf x}^\prime, {\bf d}^\prime)d{\bf x}^\prime dS_{{\bf d}}^\prime + \zeta
 %\end{equation}
 %where ${\bf r}^{(i,j)} = \x^i - \x^j$, $\mathcal{S}^2$ is the unit sphere,  and $|V_L|=L^3$ is the volume of the support of $P$, $\text{supp}(P)$, where the $N$ particles reside.  $\zeta$ is the error in the mean field approximation due to fluctuations (deviations from the mean field).  
 This allows one to replace all the sums representing interactions by integrals with respect to a probability density function $P (t, {\bf x}, {\bf d})$ of finding a given bacterium at position ${\bf x}$ with orientation ${\bf d}$.    
 %In \eqref{llneqn}, $A({\bf r},{\bf d})$ can be either the fluid velocity ${\bf u}^j({\bf r})$, the vorticity ${\omega}^j({\bf r})$,  the rate of strain ${\bf E}^j({\bf r})$ (the collisional term ${\bf F}^j$ will not contribute to the final calculations).  
 
 %\begin{remark}
 %Note in equation \eqref{llneqn}, according to the Generalized Law of Large Numbers, a factor of $N$ should appear in front.  To simplify notation we build this factor of $N$ into $P$ and enforce that $\int_{V_L} \int_{\mathcal{S}^2} P d{\bf x} dS_{{\bf d}} = N$ instead of 1.
 %\end{remark}
 % Introduce the following measure space $\left(\mathbb{R}^3 \times \mathcal{S}^2 , \mathcal{B}(\mathbb{R}^3 \times \mathcal{S}^2), \mu_P\right)$ where 
 %%P = NP' where P integrates to N and P' integrates to 1
 %$A:\left(\mathbb{R}^3 \times \mathcal{S}^2\right) \to \mathbb{R}^3$ is a random variable and $\mu_p$ is the measure  associated with the density $P({\bf x},{\bf d})$ and is defined as
 %\begin{equation*}
 %\mu_P(X) = \int_{X} P({\bf x},{\bf d}) d{\bf x}dS_{{\bf d}}, \quad X \in \mathcal{B}(\mathbb{R}^3 \times \mathcal{S}^2).
 %\end{equation*}
 By replacing the sums with integrals in the system \eqref{rij-0}-\eqref{dij} and enforcing conservation of probability, a standard Fokker-Planck equation describing the evolution of the density $P$ is obtained 
 \begin{equation}\label{init}
 \partial_t P + \nabla_{\bf x} \cdot ({\bf V} P)+ \nabla_{\bf d} \cdot (\boldsymbol{\Omega} P) - D\Delta_{\bf d} P=0,
 \end{equation}
 where the translational and orientation fluxes are defined by
 { \begin{eqnarray}
 	\hspace{-.2in} \V({\bf x},{\bf d}) &:=& V_0{\bf d} + \frac{1}{|V_L|}\int_{S^2}\int_{V_L} {\bf u} P(\x',{\bf d}')d\x'dS_{{\bf d}'}+{\bf u}^{BG}({\bf x}),\label{ktrij}\\
 	\hspace{-.2in} \boldsymbol{\Omega}(\x,{\bf d}) &:=& \frac{1}{|V_L|}\int_{S^2} \langle \boldsymbol{\omega}+B\E,P(\x',{\bf d}') \rangle_{\x'} dS_{{\bf d}'}+\boldsymbol{\omega}^{BG}({\bf d})+B\E^{BG}({\bf d}).\label{ktdij}
 	\end{eqnarray}}
 Here $<\cdot , \cdot >$ denotes the duality with respect to the $L^2$-norm, $V_L := [-L,L]^3$, and we neglect the Lennard--Jones term $F$ due to the fact that collisions only play a small role at low concentrations. The functions $u,\boldsymbol{\omega},$ and $\E$ under the integral sign depend on $\x-\x',{\bf d}$ and ${\bf d}'$, and they are defined as follows
\begin{eqnarray}
&{\bf u}({\bf x}, {\bf d}):=\frac{U_0}{8\pi\eta_0}\nabla_{\bf x}\cdot \left[({\bf d}{\bf d}-I/3)\mathcal{G}({\bf x})\right],&\nonumber\\
&\boldsymbol{\omega}({\bf x},{\bf d},{\bf d}'):=-\frac{1}{2}\bf d\times\left[\nabla_{\bf x}\times {\bf u}({\bf x},{\bf d}')\right],\label{defn}&\\
&\;\E({\bf x},{\bf d},{\bf d}'):=-{\bf d} \times\left[{\bf d} \times D_{\bf x}({\bf u}(\x,{\bf d}')){\bf d}\right],& \nonumber
 	\end{eqnarray}
 where $D_{\x}({\bf u}) := \frac{1}{2}(\nabla_\x {\bf u} + [\nabla_{\x} {\bf u}]^T)$ represents the symmetric gradient and $I$ is the identity matrix.  Also, $\boldsymbol{\omega}^{BG}({\bf d})$ and $\E^{BG}({\bf d})$ are defined in the same way as \eqref{defn}, but with the fluid velocity ${\bf u}$ replaced with the background flow ${\bf u}^{BG}$.   
 
 \begin{remark}
 	Since $\boldsymbol{\omega}, {\bf E} \sim \frac{1}{|{\bf x}-{\bf x}'|^3}$, the integrals with respect to the spatial variables must be considered in the distributional or principal value sense (which are equivalent here).  Namely,  
  \begin{equation*}
  <\frac{\partial u_i}{\partial x_j},\varphi>=C_{ij}({\bf d})\varphi(0)+\int\limits\frac{\partial u_i}{\partial x_j}(\varphi(\x)-\varphi(0))d{\x},
  \end{equation*} 
  where
 \begin{equation*}
C_{ij}({\bf d})=\lim\limits_{\varepsilon \rightarrow 0}\int\limits_{|\x|=\ve}u_in_jds_{\x}.
 \end{equation*}
 	
\end{remark}
 
 The orientation vector ${\bf d} \in \mathcal{S}^2$ can be represented by two independent angles in spherical coordinates
 \begin{eqnarray}\label{d_spherical}
 	&{\bf d} :=(\cos\alpha \sin\beta,\sin\alpha \sin\beta,\cos\beta)=(d_1,d_2,d_3),
 \end{eqnarray}
 for azimuthal angle $\alpha \in [0,2\pi)$ and polar angle $\beta \in [0,\pi)$ with unit basis vectors $\hat{\alpha} :=(-\sin\alpha,\cos\alpha,0)$ and $\hat{\beta} :=(\cos\alpha\cos\beta,\sin\alpha\cos\beta,-\sin\beta)$ respectively.
 %=1/\sin\beta(-d_y,d_x,0)
 Here one must be careful to note that the divergence and the Laplacian in orientations (the Laplace-Beltrami operator) in \eqref{init} are taken over the unit sphere.  In particular, for any field $A=A({\bf d})$ the following definition holds
 \begin{eqnarray}\nonumber
 \nabla_{{\bf d}}\cdot A &:=&\frac{1}{\sin\beta}\left[\partial _{\alpha}(A_{\alpha})+\partial_{\beta}(\sin\beta A_{\beta})\right] \\ &\hspace{5 pt}=& \tilde{\nabla}_{{\bf d}}\cdot A-\left.\frac{\partial }{\partial |{\bf d}|}\left\{|{\bf d}|^2(A\cdot {\bf d})\right\}\right|_{|{\bf d}|=1},\label{div2}
 \end{eqnarray}
 where $A_{\alpha}=A\cdot\hat{\alpha}$, $A_{\beta}=A\cdot \hat{\beta}$, and $\tilde{\nabla}_{{\bf d}}$ is the classical gradient.   
 
%In the next section we define the effective viscosity for a suspension of point dipoles.

 \subsection{\bf Definition of the effective viscosity for a suspension of point force dipoles}

 To define the effective viscosity consider the contributions to stress:  (i)  due to dipolar hydrodynamic interactions 
 $$
 \Sigma_{lm}^d(\overline{\bf d}) := \sum_{i = 1}^N \frac{U_0}{|V_L|}(d_ld_m - \delta_{lm}/3),\;\;l,m=1,2,3,
 $$
  depending only on each particle's orientation \cite{Bat70} and  (ii) due to soft collisions (the excluded volume constraints)  
  $$
  \Sigma^{LJ}_{lm}(\overline{\bf x}) := \sum_{i=1}^N \sum_{j \neq i} \frac{ F_l({\bf x}^i-{\bf x}^j) (x_m^i-x_m^j)}{|V_L|},\;\;l,m=1,2,3,
  $$
  depending only on the relative positions of each bacterium \cite{Ziebert}.  Both are combined to form the total stress due to interactions first used in \cite{RyaHaiBerZieAra11,RyaBerHaiKar12}.  We assume that all bacteria are in the volume $V_L$ at any instant of time. The bacterial configurations are denoted by $\overline{\bf x} := (\x^1,...,\x^N)$ and $\overline{\bf d} := ({\bf d}^1,...,{\bf d}^N)$.  
  
The ultimate goal is to compute the effective viscosity due to hydrodynamic interactions at low concentrations for comparison with experimental observation \cite{Sokolov} and numerical simulations.  At lower concentrations $\phi$, where the striking experimental decrease in the effective viscosity was observed, the contribution due to collisions is relatively small and for the proceeding analysis will be neglected
 \begin{equation}\label{eqn3141}
 \Sigma({\bf x}, {\bf d}) = \Sigma^d({\bf d}) + \Sigma^{LJ}({\bf x}) \approx \Sigma^d({\bf d}), \quad \text{for $\phi$ small}.
 \end{equation}
 The exact concentration interval where the formula \eqref{eqn3141} works well will be determined later by comparison with direct numerical simulations of the suspension.      
 
 Thus, it is sufficient to restrict attention to the density of orientations denoted $P_{{\bf d}}({\bf d})$ defined as  
 %In order to isolate this density,  {\it assume} a ``separation of variables", (i.e., the positions and orientations are decorrelated at low concentrations)
 %\begin{equation*}
 %P({\bf x}, {\bf d}) = P_x({\bf x})P_{{\bf d}}({\bf d}),
 %\end{equation*}
 %where $\int_{V_L} \int_{\mathcal{S}^2} P({\bf x}, {\bf d}) d{\bf x}dS_{{\bf d}} = N$.  This is one of three key assumptions used to derive the formula for the effective viscosity and is described in greater detail in Section~\ref{assumption}.  This allows for the definition of the orientation density
 \begin{equation}
 P_{{\bf d}}({\bf d}) := \frac{1}{N} \int_{V_L} P({\bf x}, {\bf d}) d{\bf x}, \qquad  \text{where} \quad \int_{\mathcal{S}^2} P_{{\bf d}}({\bf d}) = 1.
 \end{equation}
 For comparison with experiment, the main quantity of interest is the shear viscosity or component $\eta_{1212}$ of the fourth order viscosity tensor relating the stress to the strain, henceforth, denoted as $\hat{\eta}$.   We define the effective viscosity as the averaged ratio of the corresponding components of the stress and strain tensors
 \begin{equation}
 \frac{\hat{\eta} - \eta_0}{\eta_0} := \frac{1}{|V_L|} \int_{V_L} \int_{\mathcal{S}^2} \frac{\Sigma_{xy}}{\gamma} P({\bf x}, {\bf d}) d{\bf x}dS_{{\bf d}} = \frac{\rho}{\gamma}\int_{\mathcal{S}^2} \Sigma_{xy}^d({\bf d}) P_{\bf d}({\bf d})dS_{{\bf d}},\label{evsep}
 \end{equation}
 as in \cite{RyaHaiBerZieAra11,RyaBerHaiKar12}. Here $\rho=N/|V_L|$ is the mean concentration or number density. %can be defined as
 %\begin{equation}
 %\rho := \frac{N}{|V_L|} = \frac{1}{|V_L|} \int_{V_L} \int_{\mathcal{S}^2} P({\bf x}, {\bf d}) d{\bf x}dS_{{\bf d}}.
 %\end{equation}
 
  The following nonlinear, nonlocal integro-differential equation describes the evolution of the orientation density $P_{{\bf d}}(t,{\bf d})$
  \begin{equation}\label{cons2}
  \partial_t P_{{\bf d}}(t,{\bf d}) = -\nabla_{\bf d} \cdot \left(< {\bf \Omega} >_{\bf x}P_{\bf d}(t,{\bf d})\right),
  \end{equation}
  where  $< {\bf \Omega} >_{\bf x} = \frac{1}{N}\int_{V_L} {\bf \Omega} P_{\bf x}(t,{\bf x}) d{\bf x}$, ${\bf \Omega}$ contains the background flow and interaction terms 
  {\begin{equation*}
  	{\bf \Omega }(t,\x,{\bf d}) =  \boldsymbol{\omega}^{BG} +{\bf E }^{BG} + \frac{1}{N|V_L|} \int_{\mathcal{S}^2} \int_{V_L} \langle \boldsymbol{\omega}  + B{\bf E} ,  P (t,{\bf x}^\prime,  {\bf d}^\prime)  \rangle d {\bf x}^\prime  d {\bf d}^\prime.
  	\end{equation*}}
  Equation \eqref{cons2} is obtained by integrating \eqref{init} in ${\bf x}$ and dividing by $N$.

 \begin{remark}\label{remark1-ev}
 	%This effectively reduces the problem to the orientation distribution, which is the only quantity needed for computing the effective viscosity at low concentrations.  
 	In this work, lower concentrations of bacteria are considered where the primary contribution to the effective viscosity from interactions is the dipolar component of the stress, $\Sigma^d$, which only depends on the set of bacterium orientations.  Thus, the $\dot{\bf x}$ equation will not factor into the final formula; however, ${\bf F}$ is the force associated to a truncated Lennard-Jones type potential imposing excluded volume constraints.  For more information on its definition and why it is needed for global solvability see \cite{RyaBerHaiKar12}.  This quantity still remains in the original coupled ODE system used for simulations to ensure that particles remain a finite distance apart avoiding an artificial divergence in the fluid velocity ${\bf u} \sim 1/|{\bf x}^i - {\bf x}^j|$ (see section~\ref{compsim}).
 \end{remark}

 \section{\bf Conditions imposed to derive an explicit formula for the effective viscosity}\label{assumption}
 
 To calculate the effective viscosity we impose three conditions to make the system more amenable to mathematical analysis.  
 %
 %Throughout the course of this work three assumptions are used:\\
 % {\it 1. Separation of variables on the density $P(\x,{\bf d}) = P_{\x}(\x)P_{{\bf d}}({\bf d})$.\\
 % 2. The existence of a steady state.\\
 % 3. The positional density (local concentration $P_{\x}(\x)$) is constant in the $z$-direction.}\\
 %Below further details about each assumption are provided.
 
 \subsection{\bf Separation of variables}
 
 In this paper only small concentrations are considered where collisions are not important, yet the flow of each bacterium affects all others.  The bacteria are at large distances apart and, thus, since the background flow provides the major contribution to bacterial motion, then distributions of positions and orientations become {\it essentially independent} of one another.  This can be justified from the experimental work of {\it Aranson et al.} (e.g., see \cite{Sokolov2,Aranson1}).   Henceforth, it is assumed that the positions and orientations are decoupled.   % The following assumption was used in the derivation of the explicit formula for the effective viscosity \eqref{evsep} above and is consistent with experiments at low concentration.
 
 \vspace{.1in}
 
 %\noindent{\bf Hypothesis (H1):}
 \noindent{\bf Condition (C1):}  The density $P({\bf x}, {\bf d})$ can be written as 
 \begin{equation}\label{assump_separ}
 P({\bf x}, {\bf d}) = P_{\x}({\bf x})P_{{\bf d}}({\bf d}) \qquad  \text{ (separation of variables),} 
 \end{equation}
 where
 $
 P_{{\bf d}}({\bf d}) = \frac{1}{N}\int_{V_L} P({\bf x}, {\bf d})d {\bf x} 
 $
 and $\int_{\mathcal{S}^2} P_{{\bf d}}({\bf d}) dS_{{\bf d}} = 1$.  Here $N$ is the number of bacteria, $\text{supp}(P_{\x}(\x))~\subset~V_L$, where the spatial density $P_{\x}(\x)$ can be found by $P_{\x}(\x) = \int_{\mathcal{S}^2} P(\x,{\bf d}) dS_{{\bf d}}$.%$N = \rho V_L$ $\rho = N/V_L$: Number density
 
 \vspace{.1in}
 
 This condition is used twice. %was used for two reasons. 
 First, the effective viscosity at low concentration only depends on the orientation (see Remark~\ref{remark1-ev}).  Thus, using condition (C1) an explicit equation for the evolution of the orientation distribution can be derived from \eqref{cons2}.  Second,  ${\bf V}$ formally contains diverging integrals (e.g., $\int \int {\bf F} d{\x}dS_{{\bf d}}$ since ${\bf F} \sim |\x|^{-12}$), which  will no longer be present in the equation for the orientation distribution $P_{\bf d}({\bf d})$ allowing for further mathematical analysis.  It will be observed at the end of this work that the asymptotic expansion for  $P_{{\bf d}}({\bf d})$ depends on $P_{\x}(\x)$ through the coefficients, thus all the information about spatial patterns is preserved.
 
 %\vspace{.1in}
 
 \subsection{\bf Existence of a steady state $P_{{\bf d}}({\bf d})$}
 
  %The orientation density $P_{{\bf d}}(t,{\bf d})$ solves the equation \eqref{cons2}. 
  A steady state solution to \eqref{cons2} is defined as follows
 
 \begin{definition}
 	$\hat{P}_{{\bf d}}({\bf d})$ is called a steady state solution to \eqref{cons2} if it solves
 	\begin{equation*}
 	0 = -\nabla_{\bf d} \cdot \left(< {\bf \Omega} >_{\bf x}\hat{P}_{\bf d}({\bf d})\right).
 	\end{equation*}
 	%where in addition $\frac{\partial}{\partial t}\hat{P}_{{\bf d}}({\bf d}) = 0$.
 \end{definition}
 
To compute time independent effective viscosity we impose the following condition. 
% The use of this approach has been successful in \cite{SokGolFelAra09} where a steady state distribution of particles is considered and shown to match experimental observation well in the study of large scale convective motion in bacterial suspensions.
% 
 %\begin{theorem}\label{thm1}
 %\emph{(Existence of a Steady State)} There exists a {\it nontrivial} steady state solution to \eqref{cons2}.
 %\end{theorem}
 
 \vspace{.1 in}
 
 \noindent{\bf Condition (C2):} There exists a {\it nontrivial} steady state solution to \eqref{cons2}.
 
 \vspace{.1 in}
 
 First, note that there is no trivial steady state unless $B=0$ %(easily seen from \eqref{divbg}. 
 in which case we find the uniform orientation distribution ${P}_{{\bf d}}({\bf d}) = \frac{1}{4\pi}$.  This can be obtained both in the limit as $B \to 0$ in the asymptotic results derived herein for $P_{{\bf d}}({\bf d})$ and from observing that the trivial steady state would be a constant satisfying the constraint  $\int_{\mathcal{S}^2} P_{{\bf d}}({\bf d}) dS_{{\bf d}} = 1$.   One still needs to prove the existence of a steady state in the general case $B \neq 0$. The condition (C2) can be formulated as a theorem and its proof  may be the topic of a future work.   Here we remain focused on the study of the effective viscosity.
 %Two methods for approaching this are either to prove convergence of the asymptotic series or add a diffusion term $D\Delta P$ and let $D \sim B^2 \to 0$ as $B \to 0$. 
 
 \subsection{\bf $P_{\x}(\x)$ is constant in the $z$-direction}\label{ass_a3}
 
 We assume that $P_{\x}(\x)$ is constant in $z$ for the case of the planar shear background flow under consideration in this work. This is consistent with past numerical observations by {\it Ryan et al.}~\cite{RyaHaiBerZieAra11} and experimental observation in \cite{SokGolFelAra09} since the suspension remains below any critical concentration for three-dimensional collective motion.  
 %In Ref concentration is 2 x 10^{-10}cm^{-3}, for analysis conc. $\sim$ 1 x 10^{-10}cm^{-3}
 Also, collective motion even in full 3D experiments and simulations in planar shear flow has been observed to be essentially 2D in the shearing plane \cite{SokGolFelAra09}.  Thus, following experimental observation, we assume the same. 
 
 \vspace{.1in}
 
 \noindent{\bf Condition (C3):}  The density $P_{\x}(\x)$ is constant in $z$.
 
 \vspace{.1in}
 
 The condition (C3) essentially follows from the physical setup of the {\it quasi-2D} thin film suspension.  
 %The result of this assumption is that if motion is constant in the $z$-direction, then  the Fourier transform of the spatial distribution behaves like the square root of  $\delta$-function. More precisely, 
 In Appendix \ref{pxapp} we show that the condition (C3) leads to the following representation formula for the Fourier transform of the spatial distribution $F[P_{\x}]$:
 \begin{equation}\label{ass_a3_formula}
( F[P_{\x}])^2=\delta(k_3)\hat{P}^2_{12} (k_1,k_2).
 \end{equation} 
 Here ${\bf k}=(k_1,k_2,k_3)$ is the Fourier variable, and $\hat{P}_{12} (k_1,k_2)$ is a smooth function defined in ${\bf k}$-space independent of $k_3$. 
 
 %$F[P_{\x}]^2/\pi L \sim \delta(k_3)$ (see Appendix \ref{pxapp} for the proof).
 
 \section{\bf Derivation of asymptotic expression for $P_{\bf d}$ for small $B$}\label{derivodens}
 
 In this section, an expression for the orientation distribution $P_{{\bf d}}({\bf d})$ is derived.  Since \eqref{cons2} is a nonlinear integro-differential equation it is challenging, in general, to find an analytical solution.  Thus, we look for $P_{{\bf d}}({\bf d})$ by  asymptotic expansion in the limit of small non-sphericity ($B \ll 1$).  This will allow us to apply analytical techniques and derive an expression, which will provide physical insight into the mechanisms contributing to the decrease in the effective viscosity.

 Rewrite the equation for the orientation density $P_{{\bf d}}({\bf d})$ \eqref{cons2}  as  (the argument $t$ is suppressed for simpler notation) 
 %is derived from \eqref{init}.  Begin by integrating \eqref{init} in $\x$ and divide by $N$ to obtain 
\begin{equation}\label{angleeq}
 	\partial_t P_{{\bf d}}+\nabla_{{\bf d}}\cdot \left[(\boldsymbol{\omega}^{BG}+B\E^{BG})P_{{\bf d}}\right]+ \frac{1}{N|V_L|}\int \nabla_{{\bf d}}\cdot (\hat{\boldsymbol{\Omega}} P(\x,{\bf d})) d\x=0,
 	\end{equation}
 where
 	\begin{equation}
 	\hat{\boldsymbol \Omega }(\x,{\bf d}) :=  \frac{1}{|V_L|} \int_{\mathcal{S}^2}  \langle  \boldsymbol{\omega}  + B{\bf E}  ,   P_{\bf x} ({\bf x}^\prime) \rangle_{\x'}  P_{\bf d}({\bf d}^\prime) d {\bf d}^\prime. \label{deqn}
 	\end{equation}
 Herein $\hat{\boldsymbol{\Omega}}$ will denote the component of the orientational flux ${\boldsymbol{\Omega}}$ due to interactions.  Observe that the $\boldsymbol{\omega}$ and ${\bf E}$ are functions of ${\bf x}-{\bf x}'$, ${\bf d}$, and ${\bf d}'$.  %Notice that the difficulties contained in {\bf V} such as the Lennard-Jones contribution ${\bf F}$ have been removed by considering only the equation for the orientation distribution $P_{{\bf d}}({\bf d})$.

 Using Condition (C1) defined in \eqref{assump_separ} we obtain a closed form equation for a steady state $P_{{\bf d}}({\bf d})$ (provided that $P_{\x}$ is given):
\begin{eqnarray}
 	0 &=& \nabla_{{\bf d}}\cdot \left[(\boldsymbol{\omega}^{BG}+B\E^{BG})P_{{\bf d}}({\bf d})\right] \nonumber\\ &&\hspace{20 pt}+\frac{1}{N|V_L|}\int_{V_L}  \nabla_{\bf d} \cdot \left(\hat{\bf \Omega}(\x,{\bf d},{\bf d}') P_{\x}(\x)P_{{\bf d}}({\bf d})  \right)dS_{{\bf d}'} d\x.\label{sstate}
 	\end{eqnarray}
The first term in \eqref{sstate} is the contribution due to the background planar shear flow: 
 \begin{align}
 \nabla_{{\bf d}}\cdot \left[(\boldsymbol{\omega}^{BG}({\bf d})+B\E^{BG}({\bf d}))P_{{\bf d}}({\bf d})\right] &= -\frac{3\gamma B}{2}\sin^2\beta\sin 2\alpha P_{{\bf d}}({\bf d})\nonumber\\&\quad +\frac{\gamma}{2}(1+B\cos 2\alpha)\{\partial_{\alpha}P_{{\bf d}}({\bf d})\}\label{eqn:bgflow_0}\\
 &\quad +\frac{\gamma B}{4}\sin 2\alpha\sin 2\beta \{\partial_{\beta}P_{{\bf d}}({\bf d})\}.\nonumber
 \end{align}
 The second term in \eqref{sstate} is the contribution of hydrodynamic interactions between bacteria. 
 Notice the convolution form of the nonlocal terms in the spatial variable.  In the next section, the Fourier transform will be utilized to compute quantities necessary to derive the formula for the effective viscosity.  Specifically, using tools such as Parseval's Theorem, one can take the spatial integrals and consider them in Fourier space where they will prove easier to analyze.  After using the separation of variables \eqref{assump_separ}, the density will be expressed in terms of the Fourier frequencies ${\bf k}$.
 
 The main goal for the remainder of this section is to write the system in a convenient form for using the Fourier transform.  This idea follows naturally from the aforementioned observation that all the interactions terms take the form of a convolution.  
 Introduce the Fourier transform $C({\bf k}):=F[P_{\x}]({\bf k})$:
 \begin{equation}
 P_{\x}(\x)=\frac{1}{(2\pi)^3}\int e^{i{\bf k}\cdot \x}C({\bf k})d{\bf k}.
 \end{equation}
 Define ${\bf H}(\x - \x',{\bf d},{\bf d}') := \boldsymbol{\omega}(\x-\x', {\bf d}, {\bf d}') + B{\bf E}(\x-\x', {\bf d}, {\bf d}')$, then the following equalities hold 
 \begin{equation}
 <{\bf H} \star P_{\x}, P_{\x}>_{\x}=<F[{\bf H} \star P_{\x}], F[P_{\x}]>_{{\bf k}}=<F[{\bf H}],(F[P_{\x}])^2>_{{\bf k}},
 \end{equation}
 where $\star$ and $F$ stand for convolution and Fourier transform, respectively.  The first equality is Parseval's identity and the second is the fact that the Fourier transform of a convolution is the product of Fourier transforms. Thus, one can rewrite equation \eqref{sstate} in the following form
 \begin{eqnarray}
 &&	\nabla_{{\bf d}}\cdot \left[(\boldsymbol{\omega}^{BG}+B{\bf E}^{BG})P_{{\bf d}}({\bf d})\right]\nonumber\\&&\hspace{30 pt}+\int _{S^2}\nabla_{{\bf d}}\cdot \left\{P_{{\bf d}}({\bf d})P_{{\bf d}}({\bf d}')<F[{\bf H}](F[P_{\x}])^2>_{{\bf k}}\right\}dS_{{\bf d}'}=0.\label{lioumod}
 	\end{eqnarray}
 In order to compute $F[{\bf H}]$ one must first understand how the Fourier transform acts on the fluid velocity ${\bf u}$ and its derivatives.
 
 \subsection{\bf Evaluation of Fourier transforms}\label{sovft}
 
 In order to analyze \eqref{lioumod}, an analytical expression for the Fourier transform $F[{\bf H}] = F[\boldsymbol{\omega}] + BF[{\bf E}]$ is needed.   Both terms depend on the fluid velocity ${\bf u}$ defined by \eqref{stokeseqnsingle}.
 Recall the dipolar stress 
 \begin{equation}\label{def_of_sigma}
 \Sigma(\x, {\bf d}) ={\bf D}({\bf d})\delta({\x})= U_0({\bf d} {\bf d} - I/3)\delta(\x).
 \end{equation} 
 Then the Stokes equation in \eqref{stokeseqnsingle} can be written as 
 \begin{equation}\label{Stokes}
 -\eta_0\Delta_{\x}{\bf u}+\nabla_{\x}p=\nabla_{\x}\cdot \Sigma(\x,{\bf d}), \;\;\nabla_{\x}\cdot {\bf u}=0.
 \end{equation} 
 Denote the Fourier transform of a function $f(x)$ as $$\tilde{f}({\bf k})=F\left[f\right]({\bf k})=\int e^{-i({\bf k}\cdot \x)}f(\x)d\x,$$ and compute the Fourier transform of ${\bf u}$ and the symmetric gradient $D_{\x}({\bf u})$.
 
 \begin{proposition}\label{propu}
 	Let ${\bf u}$ be a solution of \eqref{stokeseqnsingle} and let $\Sigma$ be defined by \eqref{def_of_sigma}. Then 
 	\begin{eqnarray}
 	&(i)& \tilde{\Sigma}({\bf d}') =U_0\left({\bf d}'{\bf d}'^{*}-I/3\right),\nonumber\\
 	&(ii)& \tilde{\bf u}({\bf k})=\frac{i}{\eta_0|{\bf k}|}\left(I-\frac{{\bf k}{\bf k}^{*}}{|{\bf k}|^2}\right)\tilde{\Sigma}({\bf k})\frac{{\bf k}}{|{\bf k}|},\label{fourier_u}\\
 	&(iii)& F\left[D_{\x}({\bf u})\right]=-\frac{1}{2\eta_0 |{\bf k}|^4}\left(|{\bf k}|^2\tilde{\Sigma} {\bf k}{\bf k}^{*}-2{\bf k}{\bf k}^{*}\tilde{\Sigma}{\bf k}{\bf k}^{*}+|{\bf k}|^2{\bf k}{\bf k}^{*}\tilde{\Sigma}\right).\label{fourier_du}
 	\end{eqnarray}
 	Here $*$ denotes the transpose.
 \end{proposition}
 
 \begin{proof} The part $(i)$ follows from the fact that the Fourier transform of $\delta$-function is 1.
 	
 	We split the proof of $(ii)$ into two steps: First, we find the Fourier transform of the pressure $p$, then by using the first equation in \eqref{stokeseqnsingle} we find $\tilde{\bf u}$.\\
 	\noindent{\it Step 1: Evaluation of $\tilde{p}=F[p]$}.  By taking the divergence of \eqref{Stokes} in $\x$ we obtain
 	\begin{equation}
 	\Delta_{\x}p=\nabla_{\x}\cdot(\nabla_{\x}\cdot \Sigma).\label{divStokes}
 	\end{equation}
 	Observe that 
 	\begin{equation*}
 	F\left[\Delta_{\x}p\right]=-|{\bf k}|^2\tilde{p}({\bf k}), \;\; F\left[\nabla_{\x}\cdot(\nabla_{\x}\cdot\Sigma)\right]%=\int\nabla_{\x}\cdot(\nabla_{\x}\cdot\Sigma)e^{-i{\bf k}\cdot \x}d\x
 	=\int \Sigma:\nabla_{\x}^2e^{-i{\bf k}\cdot \x}d\x=-\tilde{\Sigma}({\bf k}):{\bf k}{\bf k}^{*}.
 	\end{equation*}
 	   Substituting these formulas into \eqref{divStokes} we obtain $-|{\bf k}|^2\tilde{p}({\bf k})=-\tilde{\Sigma}({\bf k}):{\bf k}{\bf k}^{*}$, and, thus, we find an expression for the Fourier transform of the pressure $p$:
 	\begin{equation}\label{eqn:tildep}
 	\tilde{p}({\bf k})=\frac{1}{|{\bf k}|^2}\tilde{\Sigma}({\bf k}):{\bf k}{\bf k}^{*}.
 	\end{equation}
 	
 	\noindent{\it Step 2: Evaluation of $\tilde{\bf u}=F[{\bf u}]$}.  Return to Stokes equation \eqref{Stokes} and observe that 
 	\begin{eqnarray*}
 		&\eta_0 F\left[\Delta_{\x}{\bf u}\right]=-\eta_0 |{\bf k}|^2\tilde{\bf u}({\bf k}), \quad F\left[\nabla_{\x}p\right]=i{\bf k}\tilde{p}({\bf k}),& \\
 		&F\left[\nabla_{\x}\cdot \Sigma\right] =i\tilde{\Sigma}({\bf k}){\bf k}.&
 	\end{eqnarray*}
 	Using these relations one finds that $\eta_0|{\bf k}|^2\tilde{\bf u}({\bf k})+i{\bf k}\tilde{p}({\bf k})=i\tilde{\Sigma}({\bf k}){\bf k}$.  After rearranging the terms and using \eqref{eqn:tildep} we complete the proof of $(ii)$.
 	
 	To prove $(iii)$ we first observe that %the Fourier transform of the symmetric gradient $D_{\x}$ is
 	%{\small \begin{equation}
 	%\frac{1}{2}\int \left(\frac{\partial u_p}{\partial x_q}+\frac{\partial u_q}{\partial x_p}\right)e^{-i{\bf k}\cdot \x}d\x=\frac{1}{2}i\int \left(u_pk_q+u_qk_p\right)e^{-i{\bf k}\cdot \x}d\x=\frac{1}{2}i\left(\tilde{u}_pk_q+\tilde{u}_pk_q\right).\label{fourierD(u)}
 	%\end{equation}}
 	$F\left[D_{\x}({\bf u})\right]=\frac{i}{2} \left(\tilde{\bf u}{\bf k}^{*}+{\bf k} \tilde{\bf u}^{*}\right)$.  Plug the Fourier transform of ${\bf u}$ from $(ii)$ into  this expression to find
  \begin{eqnarray*}
 		F\left[D_{\x}({\bf u}) \right]&=&\frac{i}{2}\left(\tilde{\bf u}{\bf k}^{*}+{\bf k} \tilde{\bf u}^{*}\right)\\&=&-\frac{1}{2\eta_0 |{\bf k}|^2}\left((I-\frac{{\bf k}{\bf k}^{*}}{|{\bf k}|^2})\tilde{\Sigma}({\bf k}){\bf k}{\bf k}^{*}+{\bf k}{\bf k}^{*}\tilde{\Sigma}({\bf k})(I-\frac{{\bf k}{\bf k}^{*}}{|{\bf k}|^2})\right).
 		\end{eqnarray*}
 	Use the fact that $\tilde{\Sigma}$ is symmetric ($\tilde{\Sigma}=\tilde{\Sigma}^{*}$) to complete the proof of $(iii)$.
 \end{proof}
 
 \begin{remark} It is easily seen that $F\left[D_{\x}({\bf u})\right]$ does not depend on $|{\bf k}|$, since $F\left[D_{\x}({\bf u})\right]$ can be rewritten as
 	\begin{equation*}
 	F\left[D_{\x}({\bf u})\right]=-\frac{1}{\eta_0} \tilde{\Sigma} \frac{{\bf k}}{|{\bf k}|}\frac{{\bf k}^{*}}{|{\bf k}|} -\frac{2}{\eta_0}\frac{{\bf k}}{|{\bf k}|} \frac{{\bf k}^{*}}{|{\bf k}|} \tilde{\Sigma}\frac{{\bf k}}{|{\bf k}|} \frac{{\bf k}^{*}}{|{\bf k}|} +\frac{{\bf k}}{\eta_0 |{\bf k}|} \frac{{\bf k}^{*}}{|{\bf k}|}\tilde{\Sigma}.
 	\end{equation*}
 	%Thus, if ${\bf k}$ is replaced by $c{\bf k}$ for $c \in \mathbb{R}$ the Fourier transform does not change.
 	%Also, $\tilde{\Sigma}=\tilde{\Sigma}({\bf d}')$ still depends on the particle orientations.
 \end{remark}
 
 This subsection is concluded by summarizing the analytical expressions for the two main components of $F[{\bf H}]=F\left[\boldsymbol{\omega}\right]+BF\left[{\bf E}\right]$:
 \begin{eqnarray}
 \hspace{-.5in}&&F\left[{\bf E}\right]=-{\bf d}\times \left({\bf d} \times F\left[D_{\x}({\bf u})\right]{\bf d}\right) =F\left[D_{\x}({\bf u})\right]{\bf d}-{\bf d}{\bf d}^{*}F\left[D_{\x}({\bf u})\right]{\bf d}\label{FourierE}\\
 \hspace{-.5in}&&F\left[\boldsymbol{\omega}\right] = -\frac{1}{2}{\bf d} \times F\left[\nabla_{\x} \times {\bf u} \right] = -\frac{1}{2}{\bf d} \times \left[-i{\bf k} \times F[{\bf u}]\right],\label{Fourierc}
 \end{eqnarray}
 where $F[{\bf u}]$ and $F[D_{\x}({\bf u})]$ are given by Proposition \ref{propu}. % In the next section, the asymptotics for the orientation density $P_{{\bf d}}({\bf d})$ will be introduced and its terms will be computed.
 
 \subsection{\bf The form of asymptotic expansion in $B$}
 
 Recall the steady-state Liouville equation \eqref{lioumod} with the background terms substituted in:
 \begin{eqnarray}
 0&=&-\frac{3\gamma B}{2}\sin^2\beta\sin 2\alpha P_{{\bf d}}({\bf d})+\frac{\gamma}{2}(1+B\cos 2\alpha)\partial_{\alpha}P_{{\bf d}}({\bf d})\nonumber\\&&+
 \frac{\gamma B}{4}\sin 2\alpha\sin 2\beta \partial_{\beta}P_{{\bf d}}({\bf d})\nonumber\\
 &&+\frac{1}{N|V_L|}\int _{\mathcal{S}^2}\nabla_{{\bf d}}\cdot \left\{P_{{\bf d}}({\bf d})P_{{\bf d}}({\bf d}')<F[{\bf H}],(F[P_{\x}])^2>_{{\bf k}}\right\}dS_{{\bf d}'}.\label{ssliou}
 \end{eqnarray}
 We consider the asymptotic expansion in the Bretherton constant, $B\ll 1$, for the orientation distribution, $P_{{\bf d}}({\bf d})$, up to the second order:
 \begin{equation}\label{asympexp}
 P_{{\bf d}}(\alpha, \beta) = P_{{\bf d}}^{(0)}(\alpha, \beta) + P_{{\bf d}}^{(1)}(\alpha, \beta)B +  P_{{\bf d}}^{(2)}(\alpha, \beta)B^2 + O(B^3).
 \end{equation}
 
 Substituting \eqref{asympexp} into \eqref{ssliou} we get different equations at different orders of $B$. It is straightforward that $P_{{\bf d}}^{(0)}(\alpha,\beta) = \frac{1}{4\pi}$ (surface area of the unit sphere is $4\pi$) solves the equation at order $O(1)$.  We want to consider the asymptotic expansion about the uniform distribution because it has been extensively documented in theory and experiment that as the bacterium bodies become or spherical ($B \to 0$), then the distribution in angles is uniform \cite{RyaHaiBerZieAra11,Haines3}.  In the next two subsections, the linear order term $P_{{\bf d}}^{(1)}(\alpha, \beta)$ and quadratic order term $P_{{\bf d}}^{(2)}(\alpha, \beta)$ are computed.% respectively.
 
 \subsection{\bf Contribution at $O(B)$}%: Background flow}
 
 %In order to analyze the linear terms of \eqref{ssliou} consider its integral term and use the properties of divergence. 
 First, notice that $\nabla_{\bf d}\cdot \boldsymbol{\omega} (\x-\x',{\bf d},{\bf d}') =0$.
 Indeed, this follows from \eqref{div2} since  $\boldsymbol{\omega}\cdot {\bf d}=0$ and  
 the classical divergence of $\boldsymbol{\omega}$ with respect to ${\bf d}$ is zero (note that $\boldsymbol{\omega}={\bf d}\times A$, where $A=\nabla_{\x}\times {\bf u}$ does not depend on ${\bf d}$). 
 This observation implies $\nabla_{{\bf d}} \cdot F[{\bf H}] =B\nabla_{{\bf d}} \cdot F[{\bf E}]$.  
 
 Using this equality and expanding the divergence under the integral sign we rewrite \eqref{ssliou} as follows:
 \begin{align}
 0 &= \frac{\gamma}{2} \left[B\sin(2\alpha)\sin\beta\left(\cos\beta\partial_{\beta}P_{{\bf d}} - 3\sin\beta P_{{\bf d}}\right) + \left(1 + B\cos(2\alpha)\right)\partial_\alpha P_{{\bf d}} \right] \nonumber\\
 &\qquad + \frac{B}{N|V_L|} \int_{\mathcal{S}^2} P_{{\bf d}}({\bf d}')P_{{\bf d}}({\bf d})\langle\nabla_{{\bf d}} \cdot (F[{\bf E}({\bf d})])(F[P_{\x}])^2\rangle_{{\bf k}}dS_{{\bf d}'}\label{divdterm}\\
 &\qquad  + \frac{1}{N|V_L|} \int_{\mathcal{S}^2} \nabla_{{\bf d}}[P_{{\bf d}}({\bf d})]P_{{\bf d}}({\bf d}') \langle F[{\bf H}({\bf d})](F[P_{\x}])^2\rangle_{{\bf k}} dS_{{\bf d}'}.\nonumber
 \end{align}
 The first integral at $O(B)$ is%zero since it can be written as  (recall ${\bf E} \sim O(B)$)
 \begin{equation}\label{OBint}
 \frac{1}{16\pi^2 N|V_L|} \int_{\mathcal{S}^2} \langle\nabla_{{\bf d}} \cdot (F[{\bf E}({\bf d})])(F[P_{\x}])^2\rangle_{{\bf k}}dS_{{\bf d}'},
 \end{equation}
 %where the Fourier Transform is defined in \eqref{FourierE}. 
 By switching the order of integration and noting $\int_{\mathcal{S}^2} \tilde{\Sigma} dS_{{\bf d}'} = \int_{\mathcal{S}^2} U_0[{\bf d}'({\bf d}')^* - I/3]dS_{{\bf d}'} = 0$ we obtain that \eqref{OBint} is zero using \eqref{FourierE} and \eqref{fourier_du}.
 
 Since both $\nabla_{{\bf d}} [P_{\bf d}({\bf d})]$ and $B{\bf E}$  are of the order $O(B)$, the second integral in \eqref{divdterm} at $O(B)$ is 
 $
 \frac{1}{4\pi N|V_L|}\int_{\mathcal{S}^2}\nabla_{\bf d}P_{\bf d}^{(1)}({\bf d}) \langle F[\boldsymbol{\omega}](F[P_{\x}])^2\rangle_{{\bf k}} dS_{{\bf d}'}
$
 which is also zero due to $\int_{\mathcal{S}^2} U_0({\bf d}'{\bf d}'- I/3) dS_{{\bf d}'} = 0$.
 
 Thus, the integral terms do not contribute to equation \eqref{divdterm} at order $O(B)$, and it has the following form:%  following steady state equation for $P_{{\bf d}}({\bf d})$ at order $B$ is found:
 \begin{equation}
 0 = \frac{\gamma}{2} \left[-3P_{{\bf d}}^{(0)}\sin(2\alpha)\sin^2\beta + \partial_\alpha P_{{\bf d}}^{(1)}\right]. \label{simplifiedOB}
 \end{equation}
 After substituting $P_{{\bf d}}^{(0)} = \frac{1}{4\pi}$ and solving \eqref{simplifiedOB}, one finds that 
 \begin{equation}\label{pd_at_order_1}
 P_{{\bf d}}^{(1)}(\alpha, \beta) = -\frac{3}{8\pi} \sin^2\beta \cos(2\alpha).
 \end{equation}  
 {\it Since the integral terms are zeros at order $O(B)$, the contribution due to interactions does not appear at order  $O(B)$ and thus the only contribution is due to the background flow.}  %To this point
 %\begin{equation*}
 %P_{{\bf d}}(\alpha, \beta) = \frac{1}{4\pi} - \frac{3}{8\pi} \sin^2\beta \cos(2\alpha)B +  P_{{\bf d}}^{(2)}(\alpha, \beta)B^2 + O(B^3).
 %\end{equation*}
 
 It will be shown later that up to $O(B)$ the contribution to the effective viscosity by the bacteria is zero.  This will shed light on the fact that interactions are {\it necessary} to see the decrease in the effective viscosity and the background flow alone is insufficient. Note that even though this is the contribution due to the background flow the strain rate $\gamma$ is not present.  
 %This is due to the fact that up to order $O(B)$ only the background flow is present and interactions don't contribute.
   Therefore, the magnitude of the flow will not have an effect on the longtime limit of the effective viscosity at $O(B)$.  However, once the terms at the next order are computed one observes a competition develop between the background flow and the flow due to inter-bacterial interactions.  In this case the magnitude of the shear $\gamma$ becomes important.
 
 \subsection{\bf Contribution at $O(B^2)$}\label{contr_at_B2}
 
 Consider terms in \eqref{divdterm} of order $O(B^2)$:
 \begin{align}
 	0 =&  \frac{\gamma}{2} \sin(2\alpha)\sin\beta \cos\beta \partial_{\beta} P_{{\bf d}}^{(1)}({\bf d}) - \frac{3\gamma}{2}\sin(2\alpha)\sin^2(\beta)P_{{\bf d}}^{(1)}({\bf d}) \nonumber \\\quad&+ \frac{\gamma}{2}\partial_{\alpha} P_{{\bf d}}^{(2)}({\bf d}) + \frac{\gamma}{2}\cos(2\alpha)\partial_{\alpha}P_{{\bf d}}^{(1)}({\bf d})  \nonumber\\
 	\quad &+ \frac{1}{4\pi N|V_L|} \nabla_{{\bf d}} \cdot \int_{\mathcal{S}^2} \langle F[{\bf E}]F[P_{\x}]^2 \rangle_{{\bf k}} P_{{\bf d}}^{(1)}({\bf d}')dS_{{\bf d}'}\nonumber \\ \quad&+  \frac{1}{4\pi N|V_L|} \int_{\mathcal{S}^2} \nabla_{{\bf d}}[P^{(2)}_{{\bf d}}({\bf d})] \langle F[\boldsymbol{\omega}](F[P_{\x}])^2\rangle_{{\bf k}} dS_{{\bf d}'}\label{OB2-int}\\
 	\nonumber \quad &+ \frac{1}{4\pi N|V_L|} \int_{\mathcal{S}^2} \nabla_{{\bf d}}[P^{(1)}_{{\bf d}}({\bf d})] \langle F[{\bf E}](F[P_{\x}])^2\rangle_{{\bf k}} dS_{{\bf d}'} \nonumber \\\quad&+  \frac{1}{N|V_L|} \int_{\mathcal{S}^2} \nabla_{{\bf d}}[P^{(1)}_{{\bf d}}({\bf d})]P^{(1)}_{{\bf d}}({\bf d}') \langle F[\boldsymbol{\omega}](F[P_{\x}])^2\rangle_{{\bf k}} dS_{{\bf d}'}.\nonumber
 	\end{align}
 
 Denote the four integral terms in equation \eqref{OB2-int} by $\text{I}_1$, $\text{I}_2$, $\text{I}_3$ and $\text{I}_4$, respectively. The following equalities hold: 
 \begin{eqnarray}
 \text{I}_1&=&\frac{ U_0}{40\pi \eta_0 N|V_L|}\left(A\sin^2\beta\cos(2\alpha) + C\sin^2\beta\sin(2\alpha)\right),\nonumber\\
 \text{I}_2&=&\text{I}_3\;=\;0,\nonumber\\
 \text{I}_4&=&\frac{3U_0}{10\pi\eta_0 N|V_L|}D\sin(2\alpha)\sin^2\beta, \nonumber
 \end{eqnarray}  
 where constants $A$, $C$, and $D$ are defined as follows
 \begin{eqnarray}
 &A :=\frac{1}{2}\int \sin^2(2\theta)\hat{P}^2_{12}k^2dk d\theta, ~C := -\frac{1}{2}\int \sin(4\theta)\hat{P}^2_{12} k^2dk d\theta,&\nonumber
 \\&& \label{coefs_ACD}
 \\
 &D := \int \cos(\theta)\sin(\theta) \hat{P}^2_{12} k^2dk d\theta.&\nonumber
 \end{eqnarray}
 Here $\hat{P}_{12}$ is from \eqref{ass_a3_formula}, and we use spherical coordinates in the Fourier space $(k=|{\bf k}|,\theta, \phi)$.  The calculations of $\text{I}_i$ can be found in Appendix \ref{KinE-OB2}.

 After substitution of the expressions for each $\text{I}_i$, we get the following equation for $P_{\bf d}^{(2)}({\bf d})$: 
 \begin{align}
 0 &= \frac{\gamma}{2} \sin(2\alpha)\sin\beta \cos\beta \partial_{\beta} P_{\bf d}^{(1)}({\bf d}) - \frac{3\gamma}{2}\sin(2\alpha)\sin^2(\beta)P_{\bf d}^{(1)}({\bf d})\nonumber \\ &\qquad+ \frac{\gamma}{2}\partial_{\alpha} P_{\bf d}^{(2)}({\bf d}) + \frac{\gamma}{2}\cos(2\alpha)\partial_{\alpha}P_{\bf d}^{(1)}({\bf d})\label{eqn:frown}\\
 & \qquad +\frac{ U_0}{40\pi \eta_0 N|V_L|}\left(A\sin^2\beta\cos(2\alpha) + C\sin^2\beta\sin(2\alpha)\right)\nonumber\\
 & \qquad+ \frac{3U_0}{10\pi\eta_0 N|V_L|}D\sin^2\beta\sin(2\alpha).\nonumber
 \end{align}
 
 Based on the form of the equation \eqref{eqn:frown}, the following representation is used to find $P_{\bf d}^{(2)}({\bf d})$:
 \begin{equation}\label{ansatzP2}
 P_{\bf d}^{(2)}({\bf d}) = C_1\sin^4\beta\cos(4\alpha) + C_2\sin^2\beta\cos(2\alpha) + C_3\sin^2\beta\sin(2\alpha).
 \end{equation}
 In order to find each $C_i$ substitute \eqref{ansatzP2} into \eqref{eqn:frown}:
 %\begin{align*}
 %0 &= -\frac{3\gamma}{4\pi}\sin(4\alpha)\sin^2\beta\cos^2\beta + \frac{9\gamma}{8\pi}\sin(4\alpha)\sin^4\beta\\
 % & \frac{\gamma}{2}\left[-4C_1\sin^4\beta\sin(4\alpha) - 2C_2\sin^2\beta\sin(2\alpha) + 2C_3\sin^2\beta\cos(2\alpha)\right] +\frac{3\gamma}{4\pi}\sin(4\alpha)\sin^2\beta\\
 % & + \frac{ U_0A}{40\pi \eta_0 N|V_L|}\sin^2\beta\cos(2\alpha) + \left[\frac{ U_0C}{40\pi \eta_0 N|V_L|}+\frac{3U_0D}{10\pi\eta_0 N|V_L|} \sin^2\beta\sin(2\alpha)\right].
 %\end{align*}
 %Simplify and group like terms to see
 \begin{align*}
 0 &= \left[\frac{3\gamma}{8\pi} - 2\gamma C_1\right]\sin(4\alpha)\sin^4\beta  + \left[ \gamma C_3  + \frac{ U_0A}{40\pi \eta_0 N|V_L|} \right]\sin^2\beta\cos(2\alpha)\\
 &  \qquad + \left[-\gamma C_2+ \frac{ U_0C}{40\pi \eta_0 N|V_L|}+\frac{3U_0D}{10\pi\eta_0 N|V_L|} \right] \sin^2\beta\sin(2\alpha).
 \end{align*}
 Since the factors are linearly independent, each coefficient is zero and, thus, we find the $C_i$'s:
 \begin{eqnarray*}
 	&C_1 = \frac{3}{16\pi},\;\;\;C_2 = -\frac{U_0(C + 12D)}{40\gamma\pi \eta_0 N|V_L|},\;\;\;C_3 = -\frac{ U_0A}{40\gamma \pi \eta_0 N|V_L|}.&
 \end{eqnarray*}
 Using these coefficients one obtains an explicit formula for the orientation distribution up to $O(B^3)$:
 \begin{align}
 P_{\bf d}(\alpha,\beta) &= \frac{1}{4\pi} - \frac{3}{8\pi} \sin^2\beta \cos(2\alpha)B+  \biggl[\frac{3}{16\pi}\sin^4\beta\cos(4\alpha)\nonumber \\ &\qquad -U_0\frac{C + 12D}{40\gamma\pi \eta_0 N|V_L|}\sin^2\beta\cos(2\alpha) \label{formpd} \\  
 & \qquad   -\frac{ U_0A}{40\gamma \pi \eta_0 N|V_L|} \sin^2\beta\sin(2\alpha)\biggr]B^2 + O(B^3). \nonumber
 \end{align}
 Formula \eqref{formpd} is the main result of Section \ref{derivodens}. Since $A$, $C$, and $D$ contain $\hat{P}_{12}$, all the spatial information is embedded in these coefficients.  In particular,  we found the lowest order (in $B$) contribution of hydrodynamic interactions to the $P_{\bf d}({\bf d})$ occurs at $O(B^2)$. In the following section, the contribution of hydrodynamic interactions to the effective viscosity is computed as well as the change in the effective normal stress coefficients.  The combination of these two quantities will describe the total effect of hydrodynamic interactions on the rheological behavior of the bacterial suspension.
 
 \section{\bf Explicit formula for the effective viscosity}\label{evform}
 
 Using the expression for the orientation distribution, $P_{\bf d}({\bf d})$ defined in \eqref{formpd}, and the formula for the effective viscosity for dipoles in a suspension \eqref{evsep}, we compute the contribution to the effective viscosity due to interactions: % as a function of the spatial density $P_{\bf x}({\bf x})$ or its Fourier transform $F[P_{\bf x}({\bf x})]$ expressed in the constants $A > 0, C$, and $D$.  Using only the dipolar contribution to the stress, $\Sigma^d({\bf d})$, one finds 
 \begin{equation}\label{eqn:evformula}
 \eta^{\text{int}}:=\frac{\eta-\eta_0}{\eta_0} 
 %= \frac{N}{|V_L|} \int_{\mathcal{S}^2} \frac{\Sigma_{xy}^d({\bf d})}{\gamma}P_{\bf d}({\bf d}) d{\bf d} 
 =  -\frac{U_0^2B^2\rho^2 \hat{A}}{75\gamma^2 \pi \eta_0}< 0.
 \end{equation}
 where $\hat{A} = \frac{1}{N^2}A \sim O (1)$ and the equality holds up to order $O(B^3)$.  The quantity $\eta^{\text{int}}$ behaves like $\rho^2$ in concentration (cf. \cite{BatGre72} where an expansion for the effective viscosity to order two in concentration is derived for passive spheres corresponding to pairwise interactions).  As an additional check of consistency, consider the dimensions of the final quantity.  The dipole moment $[U_0] = \frac{\text{kg} \cdot \text{m}^2}{\text{s}^2}$, both the Bretherton constant $B$ and $\hat{A}$ are dimensionless, the concentration/number density $[\rho] = \frac{1}{\text{m}^3}$, the ambient viscosity $[\eta_0] = \frac{\text{kg}}{\text{m}\cdot \text{s}}$, and the strain rate $[\gamma] = \frac{1}{\text{s}}$ resulting in $\eta^{\text{int}}$ being dimensionless. %having units of viscosity $\frac{\text{kg}}{\text{m} \cdot \text{s}}$.   
 
 In addition, the orientation distribution $P_{\bf d}({\bf d})$ from \eqref{formpd}  can be used to compute the effective first and second dipolar normal stress coefficients $N_{12} = \frac{\Sigma^d_{11} - \Sigma^d_{22}}{\gamma^2}$ and $N_{23} = \frac{\Sigma^d_{22} - \Sigma^d_{33}}{\gamma^2}$ to investigate the effect of hydrodynamic interactions.   The main advantage of the mathematical model is that the computation of the effective normal stress coefficients is straightforward in contrast to experiment where its measurement can be quite complicated \cite{FriHey88}. These coefficients can provide important information about the suspension.  For example,  the ratio of the first normal stress to the viscosity determines the effective relaxation time \cite{FriHey88}.  Also, phenomena such as extrudate swelling \cite{AbdHasBir74} and secondary flow \cite{RamLei08} are important in many technological applications.   A simple calculation shows that
 \begin{eqnarray}
 N_{12} &=& \frac{\Sigma^d_{11} - \Sigma^d_{22}}{\gamma^2} = \frac{U_0\rho}{\gamma^2}\left[-\frac{2}{5} - \frac{2U_0\rho(C + 12D)}{75\gamma\pi\eta_0}B^2\right]\label{eqn:normstress1-1}\\
 N_{13} &=& \frac{\Sigma^d_{22} - \Sigma^d_{33}}{\gamma^2} = \frac{U_0\rho}{\gamma^2}\left[\frac{1}{5} + \frac{U_0\rho(C + 12D)}{75\gamma\pi\eta_0}B^2\right].\label{eqn:normstress2-1}
 \end{eqnarray}
 The approximations are valid for $B \ll 1$, so for pushers ($U_0 < 0$) $N_{12} > 0$ and $N_{23} < 0$ where as for pullers ($U_0 > 0$) $N_{12} < 0$ and $N_{23} > 0$.  Both results are consistent with the predictions in \cite{Haines3,Saintillan2} while providing additional information about the concentration dependence.  The effective normal stress coefficients grow linearly with concentration in the presence of interacting bacteria; however, the fact that the normal stresses of active suspensions are non-zero in the case of a planar shear flow indicate the {\it emergence of non-Newtonian behavior}. One sees in \eqref{eqn:normstress1-1}-\eqref{eqn:normstress2-1} that as the shear rate $\gamma \to \infty$ the normal stresses approach zero indicating the dominance of the background flow on the suspension overwhelming any contribution from interactions.  % In the next subsection, the specific mechanisms that need to be present in order to observe a decrease in the effective viscosity are examined.
 
 \subsection{\bf Mechanisms required for the decrease in the effective viscosity}\label{mechev}
 
 In this subsection, the mechanisms that lead to a decrease in the effective viscosity are investigated.  These same mechanisms are shown in \cite{RyaSokBerAra13} to be responsible for collective motion and large scale structure formation in suspensions of pushers.  Our mathematical analysis provides insight beyond experiment.  Formula \eqref{eqn:evformula} reveals that elongation of bacteria, self-propulsion, and interactions are all required to observe a decrease in the effective viscosity; namely, for spherical bacteria $(B = 0)$ the net change in the  effective viscosity is zero.  In addition, active bacteria are required, since $U_0 \sim f_p = 0$ results in no change in the effective viscosity where $f_p$ is the propulsion force.  Finally, if the spatial density $P_{\bf x}({\bf x})$ is near uniform, then $\hat{A} = \frac{1}{2N^2}\int \sin^2(2\theta) \hat{P}_{12}^2 d{\bf k} \approx 0$ resulting in no change in the effective viscosity.  
 
 In the limit $\gamma \to \infty$ the contribution to motion of bacteria due to shear dominates the contribution due to interactions with $P_{\bf d}({\bf d})$ maximized at $\alpha = \pi/2$ and $\beta = \pi/2$ (alignment with $y$-axis). This is analogous to the passive case where bacteria in a planar shear flow tend to align with the direction where the fluid exerts the least amount of torque on the bacterium body. Therefore, confirming our main conclusion that in order to exhibit a decrease in the effective viscosity active, elongated bacteria whose interactions result in a non-uniform distribution in space are needed.  %Next, the emergence of stochastic behavior in the above deterministic system is described.
 
 \subsection{\bf Effective noise conjecture}\label{effnoise}
 
 In this subsection, the results herein involving a semi-dilute suspension of point force dipoles are compared to the previous result for a dilute suspension of prolate spheroids with propulsion modeled as a point force \cite{Haines3}.   Thus, the only contribution to bacterial motion is the background flow.  In \cite{Haines3}, finite size bacteria are taken as spheroids with a point force ($\delta$ function) accounting for self-propulsion.  In addition, each bacterium experiences a random reorientation referred to as tumbling.   Biologically tumbling corresponds to a reorientation of  a bacterium in hopes of finding a more favorable (nutrient rich) environment.  Typically in experiment this is observed when the concentration of oxygen is low.  Thus, bacteria enter a more dormant state resulting in a lower swimming speed and an increased tumbling rate \cite{SokAra12}. 
 
 Since only the term containing $\hat{A}$ contributes to the effective viscosity, one can choose to match the coefficient of this term 
 \begin{align*}
 P_{\bf d}^{int}&= \frac{1}{4\pi} -\frac{3}{8\pi}B\cos(2\alpha)\sin^2\beta  +\frac{3}{16\pi}B^2\sin^4\beta\cos(4\alpha)\\
 \quad &   -U_0\rho\frac{C + 12D}{40\gamma\pi \eta_0}B^2\sin^2\beta\cos(2\alpha) -\frac{ U_0\rho \hat{A}}{40\gamma \pi \eta_0} B^2\sin^2\beta\sin(2\alpha)+ O(B^3)
 \end{align*}
 with the corresponding coefficient in the derivation by {\it Haines et al.} \cite{Haines3}, which is quadratic in the diffusion strength $D$.  To make the formulas for the effective viscosity identical, the strength of the effective noise/diffusion (tumbling) is chosen to be
 \begin{equation*}
 \hat{D} := \frac{-15\eta_0\gamma^2 + \sqrt{225\eta_0^2\gamma^4 - \hat{A}^2B^2\gamma^2\rho^2U_0^2}}{12\hat{A}B\rho U_0} > 0,
 \end{equation*}
 (since $U_0 < 0$ for pushers).  Observe that $\hat{D}$, chosen in this way, depends only on the physical parameters present in the problem and the same effective viscosity as the dilute case studied in \cite{Haines3} is found.   This $\hat{D}$ is referred to as the \emph{effective noise} and the phenomenon where stochasticity arises from a completely deterministic system is called \emph{self-induced noise}.  A future work may seek to explain this phenomenon rigorously using mathematical analysis. One heuristic idea is that the periodic (deterministic) Jeffrey orbits are destroyed by interactions resulting in stochastic behavior.
 
 Some conclusions about this effective noise can be made that ensure its consistency with physical reality.  As bacteria become spheres $B \to 0, \hat{D} \to 0$ resulting in no change in the effective viscosity consistent with \cite{Haines3}.  Also as the strain/shear rate $\gamma \to \infty, \hat{D} \to 0$.  This is physically intuitive, because as the shear rate becomes large its contribution dominates that due to hydrodynamic interactions resulting in behavior that resembles that of a passive suspension.  Thus, the contribution to the effective viscosity due to hydrodynamic interactions is zero.  Finally, we compare our results with direct simulations for the coupled PDE/ODE system composed of Stokes PDE \eqref{stokeseqnsingle} and \eqref{rij-0}-\eqref{dij}.
 
 \subsection{\bf Comparison to numerical simulations}\label{compsim}

 In this section, the accuracy of the derived formula is tested by comparing it to recent numerical simulations.  The numerical procedure is outlined in \cite{RyaHaiBerZieAra11}.  These simulations are parallel in nature allowing them to be carried out on GPUs for greater efficiency.
 %Details of the simulation methods can be found in Section~\ref{sec:sims}.  All simulations herein were run on computers at Argonne National Laboratory.
 
 \begin{figure}[h!]
 	\centering
 	\includegraphics[width=2.5in]{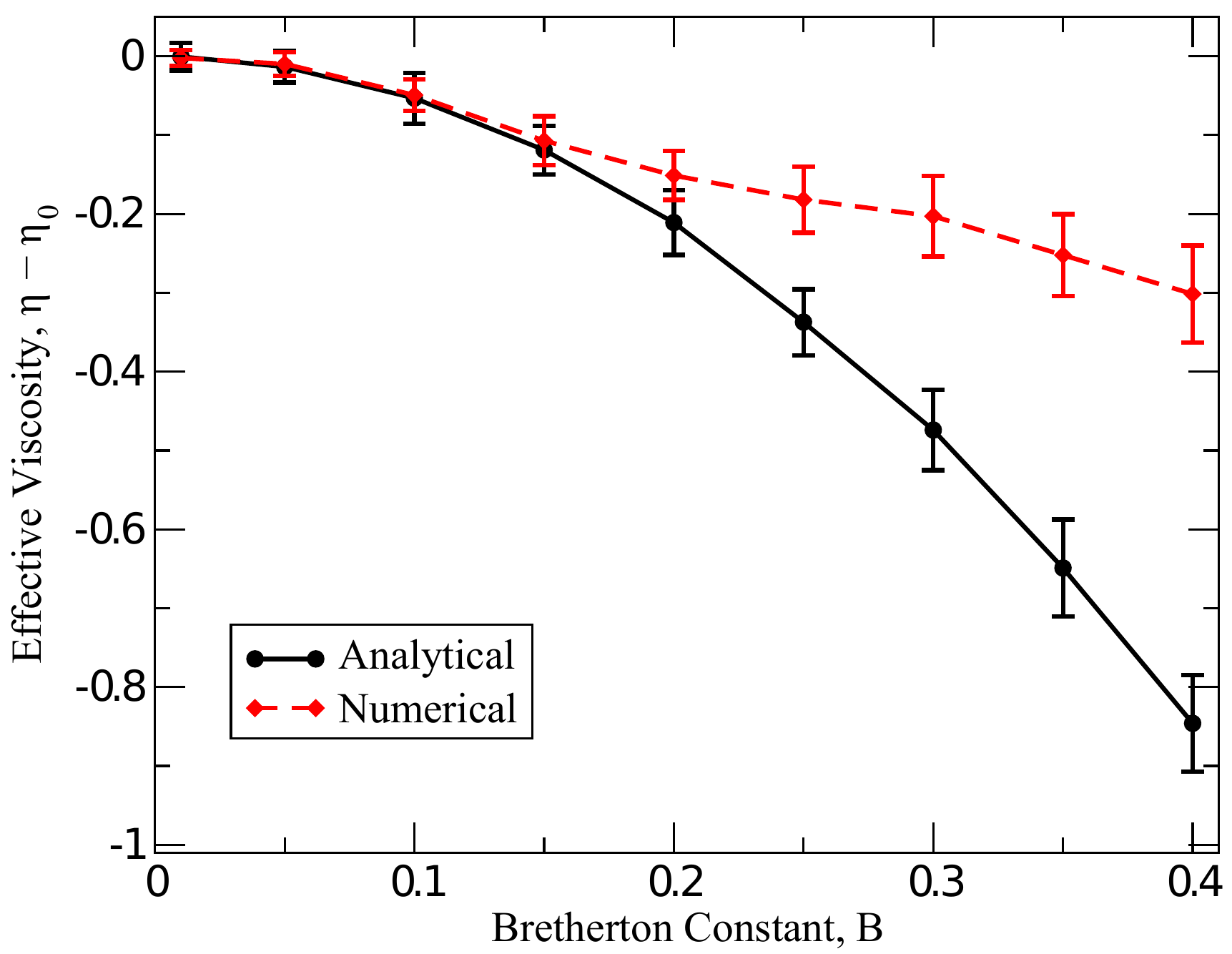}
 	\caption{\footnotesize Comparison of the formula for the effective viscosity with numerical simulations as bacterium shape changes through the Bretherton constant $B$ for a fixed volume fraction $\Phi = .02$ and shear rate $\gamma = .1$.  The vertical bars represent the error in the numerical approximation.  Error in the analytical solutions comes from the numerical estimation of $\hat{A}$.}
 	\label{fig2a-1}
 \end{figure}
 
 Figure~\ref{fig2a-1} shows how both the formula and numerical computations of viscosity change with bacterium shape as all other system parameters remain fixed.  Here shape is accounted for through the Bretherton constant $B = \frac{b^2 - a^2}{b^2+a^2}$ where $b$ is the length of the major axis and $a$ is the length of the minor axis of the ellipsoid representing a bacterium.   First, notice that in both the formula and numerics the contribution to the effective viscosity  due to hydrodynamic interactions decreases with $B$ (increasing in magnitude).  This is due to the fact that as bacteria become more asymmetrical as $B \to 1$ the inter-bacterial hydrodynamic interactions have a greater effect on alignment.  This alignment increases the magnitude of the dipolar stress leading to an even bigger decrease in the effective viscosity.  The agreement between the analytical formula and numerical simulations breaks down as $B$ becomes large, but this is expected due to the fact that the asymptotic formula is valid in the regime where $B \ll 1$ (small non-sphericity).
 
 Figure~\ref{fig2a-2} shows how both the formula and numerical computations of viscosity change with the concentration of the suspension as all other system parameters remain fixed.   It is seen that as concentration increases the effective viscosity decreases.  This can easily be explained by the fact that as the concentration increases, the motion of bacteria begins to be dominated by inter-bacterial hydrodynamic interactions.  This leads to collective motion of the bacteria in the suspension, which subsequently decreases the viscosity. The two results begin to diverge near volume fraction $\Phi \approx .02$.  The reason the numerical simulations do not decrease as much is that collisions are taken into account.  It was shown in \cite{RyaHaiBerZieAra11} that the stress due to collisions is a positive contribution to the effective viscosity that is not captured by the formula.  This contribution begins to become important beyond the dilute regime ($\Phi > 2\%$).
 
 \begin{figure}[h!]
 	\centering
 	\includegraphics[width=2.5in]{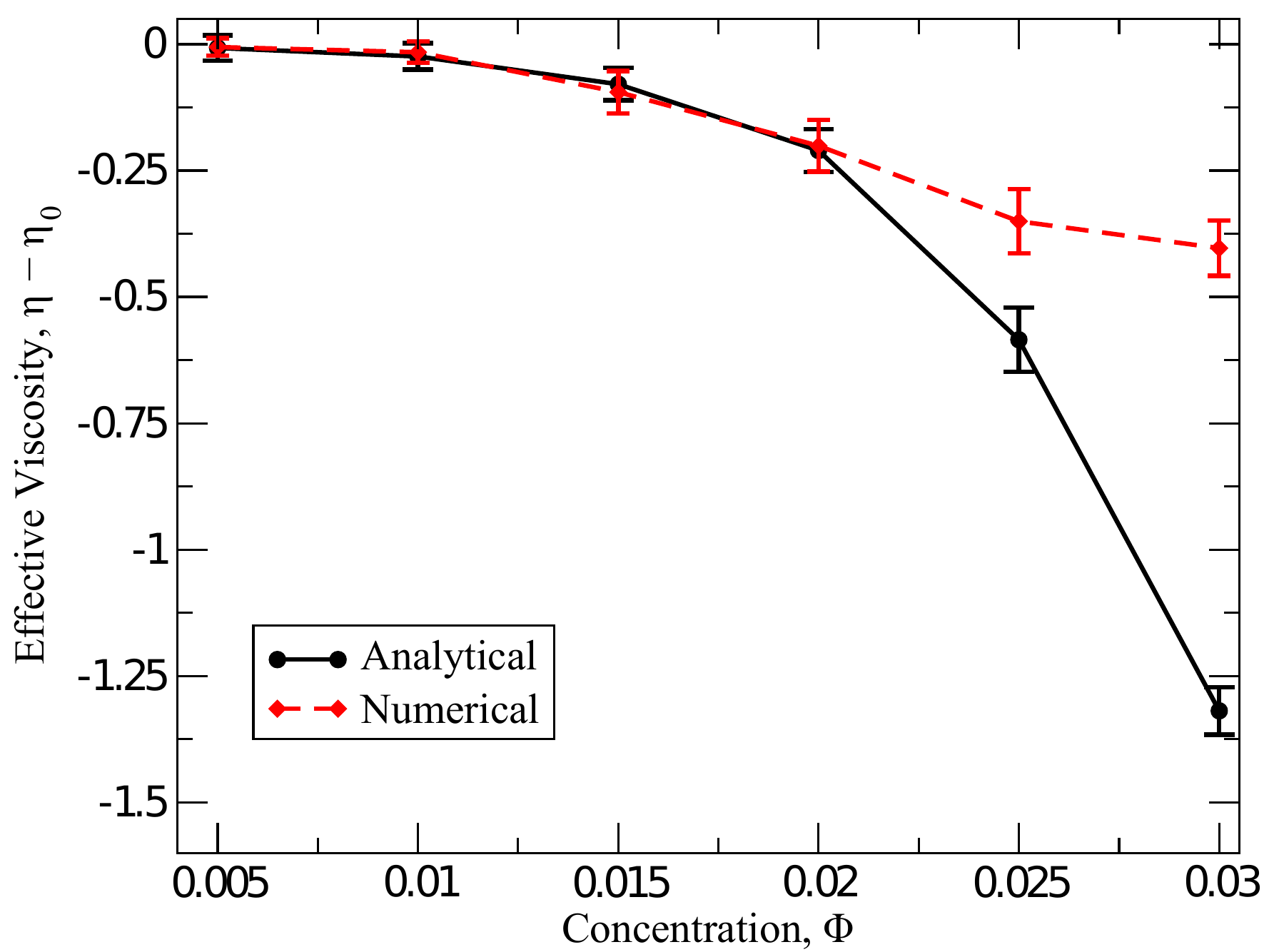}
 	\caption{\footnotesize Comparison of the formula for the effective viscosity with numerical simulations as the volume fraction $\Phi$ changes for a fixed shape $B = .2$ and shear rate $\gamma = .1$.}
 	\label{fig2a-2}
 \end{figure}
 
 Figure~\ref{fig2a-3} shows how both the formula and numerical computations of viscosity change with the shear rate of the background flow in the suspension as all other system parameters remain fixed.   As expected when the shear rate is large in both the analytical formula and simulations, the decrease in viscosity due to hydrodynamic interactions is negligible.  This is due to the fact that the background flow dominates motion of bacteria wiping out the effects of inter-bacterial interactions and stopping any collective structures from forming.  When the shear rate is too small the effective viscosity becomes unbounded.  This makes sense given that at small shear rate the system becomes almost non-dissipative and thus the effective viscosity is not well-defined.  This can easily be seen by noting that the viscosity is the ratio of the stress over the strain and when the strain is essentially zero the effective viscosity becomes  unbounded.  All three plots show good qualitative agreement with each other, experimental observation, and physical intuition.
 
 \begin{figure}[h!]
 	\centering
 	\includegraphics[width=2.5in]{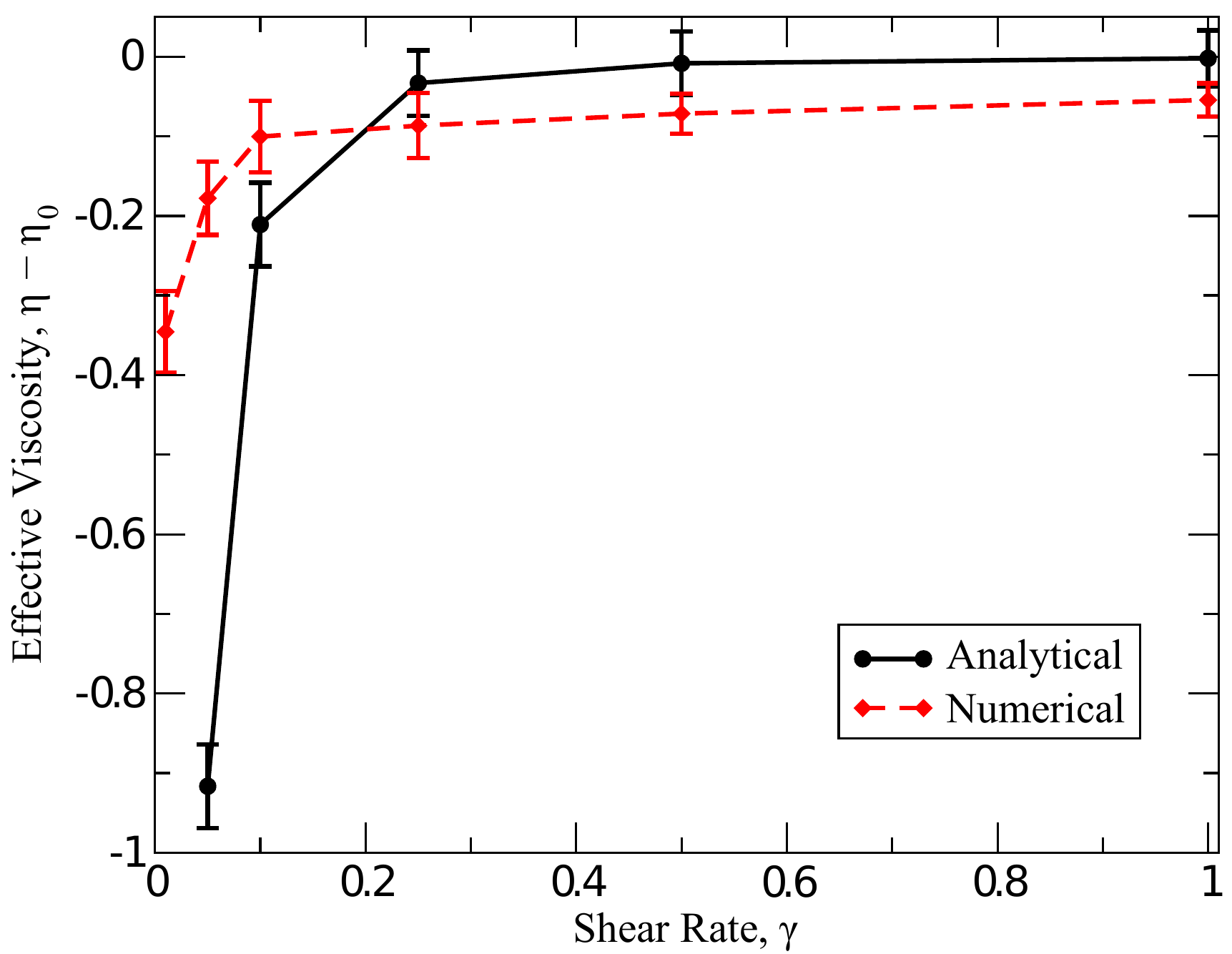}
 	\caption{\footnotesize Comparison of the formula for the effective viscosity with numerical simulations as the shear rate $\gamma$ changes for a fixed volume fraction $\Phi = .02$ and shape $B=.2$.}
 	\label{fig2a-3}
 \end{figure}

 \section{\bf Global solvability of the kinetic equation}\label{kin_solvable}
 
In this section, we study solvability of the main nonlinear integro-differential equation \eqref{cons2} governing the evolution of the orientation distribution. Primarily we are interested in existence, uniqueness, and the regularity properties of solutions of \eqref{cons2}. 

First, we note that \eqref{cons2} is an equation of the form: 
 \begin{equation}\label{fpeqn-nodiff}
 \hspace{-.3in}\partial_t P_{\bf d} = -\nabla_{{\bf d}} \cdot \left(\left[ \int_{\mathcal{S}^2} K({\bf d}, {\bf d}')P_{{\bf d}}({\bf d}') dS_{{\bf d}'}+{\bf k}({\bf d})\right] P_{{\bf d}}\right)+ D\Delta_{{\bf d}} P_{{\bf d}}.
 \end{equation}
 Indeed, one can obtain \eqref{cons2} by substituting
 \begin{eqnarray}
 K({\bf d}, {\bf d}')=\boldsymbol{\omega}({\bf d},{\bf d}') + B{\bf E}({\bf d},{\bf d}'),\;\;
 {\bf k}({\bf d})= \boldsymbol{\omega}^{BG}({\bf d}) + B{\bf E}^{BG}({\bf d}).\label{def_of_k}
 \end{eqnarray}
 Both $K$ and ${\bf k}$ from \eqref{def_of_k} are infinitely smooth functions of ${\bf d}$. Therefore, in this section we consider \eqref{fpeqn-nodiff} for the general case of smooth $K$ and ${\bf k}$.
 
 We follow the standard procedure for the analysis of the well-posedness of the evolution PDEs (e.g., see \cite{Eva98,Lio69,FroLiu12}). In particular, we introduce the notion of a weak solution. By $H^s$ ($s\in \mathbb R$) we denote the corresponding Sobolev spaces.
 
 \begin{definition}
 	For $T >0$, the function $f$  which belongs to space $\mathcal{H}$ given by \begin{equation}
 	\label{weak_class}
 	\mathcal{H}=L^2((0,T), H^1(\S)) \cap H^1((0,T),H^{-1}(\S))
 	\end{equation} is  {\it a weak solution} of \eqref{fpeqn-nodiff} if for almost all $t \in [0,T]$ and all $h \in H^{1}(\S)$
 	\begin{equation}\label{eqnc}
 	\langle \partial_t f, h \rangle = -D\langle \nabla_{\bf d} f, \nabla_{\bf d} h \rangle + \langle f, \left[\int_{\S} K({\bf d}, {\bf d}')f dS_{{\bf d}'} + {\bf k}({\bf d})\right] \cdot \nabla_{\bf d} h \rangle,
 	\end{equation}
 	where $\langle \cdot, \cdot \rangle$ is the duality product for distributions on the unit sphere $\S$.% (e.g., see \cite{FroLiu12}).  
 \end{definition}
   
   \begin{remark} 
   	According to the well-known embedding (see \cite{Sim87}) the fact that  a weak solution $f$ belongs to $\mathcal{H}$ implies that it is continuous with respect to $t\in [0,T]$ with values in $L^2(\mathcal{S}^2)$, i.e.,  $f\in C([0,T];L^2(\mathcal{S}^2))$.
   	\end{remark}
   
 \begin{definition}
 	A function $f\in C([0,T];L^2(\mathcal{S}^2))$ is called {\it positive in distributional sense} if  
 	\begin{equation}
 	\langle f,h\rangle \geq 0 
 	\end{equation}
 	for all $t \in [0,T]$ and all $h \in C(\mathcal{S}^2)$ such that $h({\bf d})\geq 0$ for all ${\bf d}\in \mathcal{S}^2$.  
 	\end{definition}
 The following theorem is the main result of this section. 
 
 \begin{theorem}\label{thm_galerkin}
 	Assume $f_0\in L^2(\mathcal{S}^2)$, $K\in C^2(\mathcal{S}^2\times \mathcal{S}^2)$, ${\bf k}\in C^2(\mathcal{S}^2)$ and $T>0$. Assume also that $f_0$ is positive in the distributional sense. Then the following statements hold:
 	\begin{list}{}{}
 		\item [{\it (i)}] There exists the unique weak solution of \eqref{fpeqn-nodiff} $f$ on interval $[0,T]$ such that $f|_{t=0}=f_0$. The weak solution $f$ is positive. It continuously depends on initial conditions, i.e., there exists a positive constant $C>0$ such that 
 		\begin{equation}
 		\label{cont_dep_on_ic}
 		\sup\limits_{t\in[0,T]} \|f^{(1)}-f^{(2)}\|_{L^2(\mathcal{S}^2)}\leq C \|f^{(1)}_0-f^{(2)}_0\|_{L^2(\mathcal{S}^2)},
 		\end{equation}
 		where $f^{(1)}$ and $f^{(2)}$ are weak solutions with initial conditions $f^{(1)}|_{t=0}~=~f^{(1)}_0$ and $f^{(2)}|_{t=0}=f^{(2)}_0$, respectively.  
 	
 	\medskip 
 	
	 	\item [{\it (ii)}] For all $s\geq 0$ if $f_0\in H^s(\mathcal{S}^2)$, then $f\in C([0,T];H^{s}(\mathcal{S}^2))$. \\If $f_0~\in~ C^{\infty}(\mathcal{S}^2)$, then $f~\in~ C([0,T];C^{\infty}(\mathcal{S}^2))$.
	 	
	\medskip 
	 
	 	\item[{\it (iii)}] For all $s\geq 0$ if $f_0\in H^s(\mathcal{S}^2)$, then
	 	for all $m \geq 0$ and $t>0$:
	 	\begin{equation}\label{regul_ineq}
	 	\| f(t)\|^2_{H^{s+m}(\mathcal{S}^2)} \leq C\left(1 + \frac{1}{t^m}\right),
	 	\end{equation}
	 	where the constant $C$ depends only on $\| f_0\|_{H^s(\S)}$, $s$, and $m$.  In particular,
	 	$$f\in C((0,\infty); H^p(\mathcal{S}^2))$$ for all $p\in \mathbb Z$. 
	 	% for any $t_0 > 0$ and $s>0$ the solution $f$ is uniformly bounded on $[t_0, \infty)$ with respect to the $H^s$ norm.
 	\end{list}
 \end{theorem}
% The proof relies on the standard method of Galerkin approximations. 
 \begin{proof} $\;$\\
 	\noindent{STEP 0.} ({\it Preliminaries})
	 %In order to use the Poincare inequality we 
	 Consider spaces of functions with mean zero: 
	 \begin{equation*}
	 \dot{L}^2(\mathcal{S}^2):=L^2(\mathcal{S}^2)\cap\left\{f:\langle f ,1\rangle=0\right\}\;\;
	 \dot{H}^s(\mathcal{S}^2):=H^s(\mathcal{S}^2)\cap\left\{f:\langle f ,1\rangle=0\right\}.
	 \end{equation*}  
	 Note that for $f\in L^1(\mathcal{S}^2)$
	 \begin{equation*}
	 \langle f ,1\rangle=\int_{\mathcal{S}^2}f d S_{\bf d}.
	 \end{equation*}
	 We use $\|\nabla_{\bf d}f\|_{L^2(\mathcal{S}^2)}$ as a norm in $\dot{H}^1(\mathcal{S}^2)$.
	  	 
	 In this proof we assume that $\int_{\mathcal{S}^2} f_0 dS_{\bf d}=1$. Consider the ``mean zero" component of the solution $f$; namely, $g:=f-\frac{1}{4\pi}$. If $f$ is the weak solution of \eqref{fpeqn-nodiff}, then $g$ satisfies   	
	 \begin{eqnarray}
		\frac{d}{dt}\langle g, h\rangle &=& -D\langle \nabla_{\bf d} g, \nabla_{\bf d} h \rangle + \langle \frac{1}{4\pi} + g, \int_{\mathcal{S}^2} K({\bf d},{\bf d}')g ({\bf d}')dS_{{\bf d}'} \cdot \nabla_{{\bf d}} h\rangle \nonumber \\&&+ \langle \frac{1}{4\pi}+g, \left[\int_{\S} K({\bf d},{\bf d}') dS_{{\bf d}'} +{\bf k}({\bf d})\right] \cdot \nabla_{{\bf d}} h \rangle
		\label{eqn_for_g}
		\end{eqnarray}	
 	for all $h\in {H}^1(\mathcal{S}^2)$. Existence, uniqueness, and continuous dependence on initial conditions will be proven for $g$, which is equivalent to the proof of the same properties for $f$. 
 	
 	Below $C$ denotes a positive constant and it may change from line to line. 
 	\medskip 
 	
 	\noindent{STEP 1.} ({\it Local existence})
 	Let $E_N$ be the space spanned by the first $N$ eigenvalues of the Laplace-Beltrami operator $\Delta_{\bf d}$, and let $\Pi_N$ be the orthogonal projector on the space $E_N$. Introduce the Galerkin approximation $g^N$, which is the solution of the following equation:
 	\begin{eqnarray}
 	\frac{d}{dt}\langle g^N, h\rangle &=& -D\langle \nabla_{\bf d} g^N, \nabla_{\bf d} h \rangle + \langle \frac{1}{4\pi} + g^N, \int_{\mathcal{S}^2} K({\bf d},{\bf d}')g^N ({\bf d}')dS_{{\bf d}'} \cdot \nabla_{{\bf d}} h\rangle \nonumber \\&&+ \langle \frac{1}{4\pi}+g, \left[\int_{\S} K({\bf d},{\bf d}') dS_{{\bf d}'} +{\bf k}({\bf d})\right] \cdot \nabla_{{\bf d}} h \rangle,
 	\label{eqn_for_gN}
 	\end{eqnarray}	      
 	for all $h\in E_N$, and $g^N|_{t=0}=\Pi_N g_0$, where $g_0:=f_0-\frac{1}{4\pi}$.      
 	      
 	In a standard manner, the problem \eqref{eqn_for_gN} can be interpreted as a system of $N$ ODEs, and its solution $g^N$ exists for $t\in [0,t_N)$ for some $t_N>0$. Taking $h=g^N$ in \eqref{eqn_for_gN}, using the Cauchy inequality, and the boundedness of $K$ and ${\bf k}$  we obtain 
 	\begin{equation}\label{ineq_1_gN}
 	\frac{d}{dt}\|g^N\|^2_{L^2(\mathcal{S}^2)}\;+\;D\|g^N\|^2_{H^1(\mathcal{S}^2)}\leq C \left(1+\|g^N\|^4_{L^2(\mathcal{S}^2)}\right). 
 	\end{equation}      
 	In the inequality \eqref{ineq_1_gN} the constant $C$ does not depend on $N$. This implies that $g^N$ exists for $0\leq t\leq t_0$ where $t_0$ may be chosen independently from $N$, and 
 	\begin{equation}\label{bound_on_gN}
 	\|g^N(t)\|_{L^2(\mathcal{S}^2)}\leq C, \;\;0\leq t\leq t_0,
 	\end{equation} 
 	%where the constant $C$ depends on $\|g_0\|_{L^2(\mathcal{S}^2)}$. 
 	The bound \eqref{bound_on_gN} gives that the RHS of \eqref{ineq_1_gN} is estimated by a constant independent from $N$. Then by integrating \eqref{ineq_1_gN} in $t$ we get
 	\begin{equation}\label{boundH1_on_gN}
 	\int_0^{t_0} \|g^N\|^2_{H^1(\mathcal{S}^2)}dt\leq C.
 	\end{equation}      
 	Take $h \in L^2(0,t_0; \dot{H}^1(\mathcal{S}^2))$ in \eqref{eqn_for_gN}, integrate in $t$, and use the Cauchy inequality, $\langle u,v \rangle\leq C \|u\|_{H^1(\mathcal{S}^2)}\|v\|_{H^{-1}(\mathcal{S}^2)}$, and the Minkovsky inequality to obtain
 	 \begin{align*}
 			\int_0^{t_0} \langle \partial_t g^N, h\rangle dt \leq C\left[\int_0^{t_0} \| h\|_{H^{1}(\mathcal{S}^2)}^2dt\right]^{1/2}.\nonumber
 		\end{align*}
 		Therefore,
 		\begin{equation}\label{bound_on_gtN}
 			\int_0^{t_0} \| \partial_t g^N\|^2_{H^{-1}(\mathcal{S}^2)} dt \leq C. 
 		\end{equation}
 	From bounds \eqref{bound_on_gN}, \eqref{boundH1_on_gN}, \eqref{bound_on_gtN} and the following relation which holds for all $g,h$ from $\mathcal{H}$
 	\begin{equation*}
 	\int_0^{t_0} \langle \partial_t g, h\rangle= -\int_0^{t_0} \langle g,\partial_t h\rangle dt +\langle g(t_0), h(t_0)\rangle - \langle g(0), h(0)\rangle,
 	\end{equation*} 
 	we obtain that there exists $g\in\mathcal{H}$ such that (up to a subsequence) 	
 		\begin{eqnarray}
 		&g^N \rightharpoonup g& \text{ in }L^{\infty}(0,t_0; L^2(\mathcal{S}^2))\cap L^{2}(0,t_0; H^1(\mathcal{S}^2)),\label{weak_conv_g}\\
 		&\partial_t g^N \rightharpoonup g& \text{ in } L^{2}(0,t_0; H^{-1}(\mathcal{S}^2)).\label{weak_conv_gt}
 		\end{eqnarray}
 	In particular, weak convergences in \eqref{weak_conv_g} and \eqref{weak_conv_gt} imply strong convergence in $C([0,t_0];L^2(\mathcal{S}^2))$. Thus, $$g|_{t=0}=\lim_{N\to\infty }  g^N|_{t=0}=\lim_{N\to \infty} \Pi_N g_0=g_0, $$ and  
 	\begin{equation}\label{strong_conv_K}
 	\int_{\mathcal{S}^2} K({\bf d},{\bf d}')g^N ({\bf d}')dS_{{\bf d}'} \rightarrow 	\int_{\mathcal{S}^2} K({\bf d},{\bf d}')g ({\bf d}')dS_{{\bf d}'} \text{ in } C([0,t_0];L^2(\mathcal{S}^2)).
 	\end{equation}
 	To complete the proof of local existence we need to show that $g$ solves \eqref{eqn_for_g}. To this end, consider \eqref{eqn_for_gN} with $h=w(t)h_0$, where $h_0\in E_M$, $M<N$ and $w(t)$ is arbitrary smooth function of one argument $t$. Integrate this equation in $t$ over the interval $(0,t_0)$, and pass to the limit $N\to \infty$ ($M$ is fixed) using \eqref{weak_conv_g}, \eqref{weak_conv_gt} and \eqref{strong_conv_K}. Since $w(t)$ is arbitrary we obtain that \eqref{eqn_for_g} is satisfied for all $h_0$ from the space $\cup_{M} E_M$ which is dense in $\dot{H}^1(\mathcal{S}^2)$. Therefore, $g$ solves \eqref{eqn_for_g} for all $h \in \dot{H}^1(\mathcal{S}^2)$. 
 	
 	 Thus, we constructed a function $g$ that is a weak solution of \eqref{eqn_for_g} defined on the time interval $0\leq t\leq t_0$. 
 	
 	\medskip 
 	
 	 \noindent{STEP 2.} ({\it Uniqueness \& continuous dependence on initial conditions})    \\
 	Consider $g^{(1)}$ and $g^{(2)}$, weak solutions of \eqref{eqn_for_g} defined on the time interval $[0,t_0]$ with initial data $g^{(1)}_0$ and $g^{(2)}_0$, respectively. 
 	For both $i=1$ and $i=2$ if one substitutes $h=g^{(i)}$ into the equation \eqref{eqn_for_g} written for $g^{(i)}$, one obtains by using the same arguments as for \eqref{bound_on_gN} that %and \eqref{boundH1_on_gN} 
 	%it follows that 
 	\begin{equation}\label{gi_bound}
 	\sum_{i=1,2}\|g^{(i)}\|^2_{L^2(\mathcal{S}^2)}%+\int_0^{t_0}\|g^{(i)}\|^2_{H^1(\mathcal{S}^2)}dt
 	< C,\;\;0\leq t\leq t_0,
 	\end{equation}
 	where the constant $C$ depends on initial data $g^{(i)}_0$ and the parameter $D$ only. 
 	
 	By subtracting equation \eqref{eqn_for_g} written for $g^{(2)}$ from equation \eqref{eqn_for_g} written for $g^{(1)}$ we get the following equality 
 	\begin{eqnarray}
	\langle \partial_t u, h\rangle &=& -D\langle \nabla_{{\bf d}} u, \nabla_{{\bf d}} h\rangle + \langle \left[\int_{S^2} K({\bf d},{\bf d}') u dS_{{\bf d}'}\right], \nabla_{{\bf d}} h\rangle \nonumber \\&&+ \langle u, \left[\int_{S^2} K({\bf d},{\bf d}') dS_{{\bf d}'} + {\bf k}\right]\nabla_{{\bf d}} h \rangle\nonumber\\&
	& + \langle u, \left[\int_{S^2} K({\bf d},{\bf d}') g^{(1)} dS_{{\bf d}'}\right] \nabla_{{\bf d}} h\rangle  \nonumber \\&&+\langle g^{(2)}, \left[\int_{S^2} K({\bf d},{\bf d}') u dS_{{\bf d}'} \right]\nabla_{{\bf d}} h \rangle.\label{eqn4}
	\end{eqnarray}
 	By taking $h=u$, using the Cauchy inequality, and \eqref{gi_bound} we obtain 
 	\begin{equation*}
 	\frac{d}{dt}\|u\|^2_{L^2(\mathcal{S}^2)} +D \|u\|^2_{H^1(\mathcal{S}^2)}\leq C\|u\|^2_{L^2(\mathcal{S}^2)}.
 	\end{equation*} 
 	This inequality implies that $\|u(t)\|^2_{L^2(\mathcal{S}^2)}\leq e^{Ct}\|u(0)\|^2_{L^2(\mathcal{S}^2)}$, and, thus,
 	\begin{equation}\label{g_cont_in_ic}
 		\|g^{(1)}(t)-g^{(2)}(t)\|_{L^2(\mathcal{S}^2)}<e^{Ct} 	\|g^{(1)}_0-g^{(2)}_0\|_{L^2(\mathcal{S}^2)}.
 	\end{equation}
 	Again, the constant $C$ depends on initial data $g^{(i)}_0$ and the parameter $D$ only.
 	
 	The inequality \eqref{g_cont_in_ic} implies that a weak solution of \eqref{eqn_for_g} continuously depends on the initial data. In particular, uniqueness holds: if $g^{(1)}_0=g^{(2)}_0$, then from \eqref{g_cont_in_ic} it follows that the corresponding solutions  $g^{(1)}$ and $g^{(2)}$ coincide. 
 	
 	\medskip 
 	
 	\noindent{STEP 3.} ({\it Regularity of weak solutions})\\ 
 	Consider a weak solution $g$ and assume $g_0\in \dot{H}^{s}(\mathcal{S}^2)$ that $s\in \mathbb Z_+$. Such a weak solution exists due to STEP 1, and it can be approximated by Galerkin approximations $g^N$ which follows from uniqueness proved in STEP 2.    
 	
 	By substituting $h=(-\Delta_{\bf d})^{s}g^N$ into the equation \eqref{eqn_for_gN}, using the Cauchy inequality and \eqref{bound_on_gN} we obtain
 	\begin{equation}\label{ineq_g_s}
 	\frac{d}{dt}\|g^N\|^2_{H^s(\mathcal{S}^2)}+D \|g^N\|^2_{H^{s+1}(\mathcal{S}^2)}\leq C(\|g^N\|^2_{H^s(\mathcal{S}^2)}+1),
 	\end{equation} 
 	where the constant $C$ depends on $\|g_0\|_{L^2(\mathcal{S}^2)}$, $\|g_0\|_{H^s(\mathcal{S}^2)}$ and the parameter~$D$. In the same manner as for \eqref{bound_on_gN}, \eqref{boundH1_on_gN} and \eqref{bound_on_gtN} it follows from \eqref{ineq_g_s} that 
 	\begin{eqnarray*}
	 g^N && \text{ is bounded in }L^2 (0,t_0; H^{s+1}(\mathcal{S}^2))\cap L^{\infty} (0,t_0; H^{s}(\mathcal{S}^2)),\\
	 \partial_t g^N && \text{ is bounded in }L^2 (0,t_0; H^{s-1}(\mathcal{S}^2)).   
 	\end{eqnarray*} 
 	Hence, $g\in L^2 (0,t_0; H^{s+1}(\mathcal{S}^2))\cap L^{\infty} (0,t_0; H^{s}(\mathcal{S}^2)) \cap H^1(0,t_0; H^{s-1}(\mathcal{S}^2))$. The standard embedding theorem (e.g., from \cite{Sim87}) implies $g\in C([0,t_0];H^{s}(\mathcal{S}^2))$. 
 	
 	\medskip 
 	
 	\noindent{STEP 4.} ({\it Positivity of weak solutions})\\ 
 	Consider $f=\frac{1}{4\pi}+g$, a weak solution of \eqref{fpeqn-nodiff}. Assume first $f_0\in H^4(\mathcal{S}^2)$ and $f_0({\bf d})\geq 0$. Then $f$ belongs to $C([0,t_0];C^2(\mathcal{S}^2))$, and thus $f$ is a classical solution of~\eqref{fpeqn-nodiff}: 
 	\begin{equation*}
 	\partial_t f= D\Delta_{\bf d} f-F\cdot \nabla_{\bf d} f-(\nabla_{\bf d} \cdot F)f,  
 	\end{equation*}  
 	where $F({\bf d}):=\int_{\mathcal{S}^2}K({\bf d},{\bf d}')f({\bf d}')dS_{{\bf d}'}+{\bf k}({\bf d})\in C([0,t_0];C^1(\mathcal{S}^2))$.
 	Consider $\tilde{f}:=fe^{\omega t}$, where $\omega:=\max\limits_{[0,t_0]\times \mathcal{S}^2}|\nabla_{d}\cdot F|$. Then $\tilde{f}$ solves the following equation 
 	\begin{equation*}
	\partial_t \tilde{f}= D\Delta_{\bf d} \tilde{f}-F\cdot \nabla_{\bf d} \tilde{f}+(\omega-\nabla_{\bf d} \cdot F)\tilde{f}. 	
 	\end{equation*}  
 	Since $\omega-\nabla_{\bf d} \cdot F\geq 0$ the weak maximum principle for parabolic equations applies for $\tilde{f}$, and, thus, $f\geq 0$. 
 	
 	Consider the case of $f_0\in L^2(\mathcal{S}^2)$, which is positive in the distributional sense. Then we can approximate $f_0$ by positive $f^N_0\in H^4(\mathcal{S}^2)$ in the space $L^2(\mathcal{S}^2)$. Denote by $f^N$ solutions of \eqref{fpeqn-nodiff} with initial data $f^N_0$. Then by \eqref{g_cont_in_ic} we can pass to the limit $N\to\infty$ in the inequality
 	\begin{equation*}
 	\langle f^N(t), h\rangle \geq 0
 	\end{equation*}  
 	for all $0\leq t \leq t_0$ and $h\in C(\mathcal{S}^2)$. Thus, the function $f$, which is the solution of \eqref{fpeqn-nodiff} with initial data $f_0$, is positive at least in the distributional sense.
 	
 	\medskip 
 	
 	\noindent{STEP 5.} ({\it Global existence})\\
	Consider $f_0=\frac{1}{4\pi}+g_0\in L^2(\mathcal{S}^2)$, which is positive in the distributional sense. Functions $f$ and $g$ are weak solutions of \eqref{fpeqn-nodiff} and \eqref{eqn_for_g}, respectively. We want to prove in this step that the time interval on which $f$ and $g$ are defined can be extended from $[0,t_0]$ to $[0,T]$ for any given $T>0$.    	
 	  
 	First, observe that 
 	\begin{equation*}
 	\int_{\mathcal{S}^2}f(t) dS_{\bf d}=\int_{\mathcal{S}^2} f_0dS_{\bf d}=1.  
 	\end{equation*}  
 	From the equality above and positivity of $f$ established in STEP 4 we obtain 
 	\begin{equation*}
 	\|f(t)\|_{L^1(\mathcal{S}^2)}=1.
 	\end{equation*}
 	In particular, since $|g|\leq |f|+\frac{1}{4\pi}$ we have 
 	\begin{equation}\label{aux_ineq_for_global}
 	\int_{\mathcal{S}^2}K({\bf d},{\bf d'})g({\bf d}')dS_{{\bf d}'}\leq {C} (\|f(t)\|_{L^1(\mathcal{S}^2)}+1)=2C.
 	\end{equation}
 	
 	Substitute $h=g$ into \eqref{eqn_for_g}, use the Cauchy inequality and \eqref{aux_ineq_for_global} to obtain
 	\begin{equation*}
 	\frac{d}{dt}\|g\|_{L^2(\mathcal{S}^2)}^2 + D\|g\|_{H^1(\mathcal{S}^2)}^2\leq C \left(\|g\|_{L^2(\mathcal{S}^2)}^2+1\right). 
 	\end{equation*} 
 	Then the $L^2$-norm of the weak solution is bounded on all bounded time intervals $[0,T]$: $$\max\limits_{0\leq t\leq T}\|g(t)\|_{L^2(\mathcal{S}^2)}^2<C(e^{CT}+1).$$
 	Thus, global existence follows. 
 	
 	\medskip 
 	
 	\noindent{STEP 6.} ({\it Instantaneous regularity})\\
	Consider positive $f_0\in H^s(\mathcal{S}^2)$ and the corresponding weak solution $f=\frac{1}{4\pi}+g$ of \eqref{fpeqn-nodiff}. According to STEP 3  $f\in L^2([0,T];H^{s+1}(\mathcal{S}^2))$ and, thus, $f\in H^{s+1}(\mathcal{S}^2)$ for almost all $t>0$. Hence, there exists $\tilde{t}>0$ arbitrarily close to $0$ such that $f(\tilde{t})\in H^{s+1}(\mathcal{S}^2)$. Then by uniqueness and STEP 3, $f\in C([\tilde{t},T];H^{s+1}(\mathcal{S}^2))$. We can choose $\tilde{t}$ arbitrarily small and $T$ arbitrarily large (due to global existence proved in STEP 5). By repeating the same arguments for $s+1$, $s+2$, and so on we get 
	\begin{equation*}
	f\in C(0,+\infty;H^{p}(\mathcal{S}^2))
	\end{equation*}     
	for all $p\in \mathbb Z$.

	Next we prove \eqref{regul_ineq} by induction with respect to $m$. Substitute $h=(-\Delta_{\bf d})^{s}g+t(-\Delta_{\bf d})^{s+1}g$ for $t>0$ in \eqref{eqn_for_g} and use the Cauchy inequality to obtain
		\begin{eqnarray*}%\label{eqnstar}
&&\frac{d}{dt}\left(\| g \|^2_{\dot{H}^s(\mathcal{S}^2)} +\frac{D}{2}t\| g \|^2_{\dot{H}^{s+1}(\S)}\right)\\ && \hspace{30 pt}+ \frac{D}{2}\left(\| g \|^2_{\dot{H}^{s+1}(\S)} + \frac{D}{2}t\| g\|^2_{\dot{H}^{s+2}(\S)}\right) \leq C\| g_0\|^2_{H^s(\S)}(1+t).
	\end{eqnarray*}  
	Using the Poincare inequality $\|g\|_{\dot{H}^{s+k}(\mathcal{S}^2)}\leq\|g\|_{\dot{H}^{s+k+1}(\mathcal{S}^2)}$ we obtain 
	\begin{equation*}
	\| g \|^2_{\dot{H}^s(\mathcal{S}^2)} +\frac{D}{2}t\| g \|^2_{\dot{H}^{s+1}(\S)}\leq %C\|g_0\|_{\dot{H}^{s}(\mathcal{S}^2)} + 
	C\| g_0\|^2_{H^s(\S)}(1+t).%+Ct.
	\end{equation*}
	 Thus, the base of induction is shown  
	 \begin{equation}\label{m1}
	 \| g \|^2_{\dot{H}^{s+1}(\S)}<C\| g_0\|^2_{H^s(\S)}\left(1+\frac{1}{t}\right).
	 \end{equation}
	
	Finally, to get the inequality \eqref{regul_ineq} at the order $m+1$ we use the inequality  \eqref{regul_ineq} at order $m$ between times $t/2$ and $t$ and \eqref{m1} between times 0 and $t/2$:
	\begin{align*}
	\| g(t)\|^2_{H^{s+m+1}(\S)} &\leq C\| g\left(\frac{t}{2}\right)\|^2_{H^{s+1}(\S)}\left( 1 + \left(\frac{2}{t}\right)^m\right)\\
	&\leq C\| g_0\|^2_{H^s(\S)}\left(1 + \frac{1}{t^{m+1}}\right).
	\end{align*}
	
	Thus, \eqref{regul_ineq} is proved by induction. 
	
	\medskip 
	
	\noindent{STEP 7.} ({\it Proof of Theorem \ref{thm_galerkin}})
	\begin{list}{}{}
	\item{({\it i})} Existence of a weak solution of \eqref{fpeqn-nodiff} for arbitrary $T>0$ is proved in STEP 5. Uniqueness is proved in STEP 2. To prove continuous dependence on initial data on arbitrary time interval $[0,T]$ one needs to repeat all arguments in STEP 2 replacing $t_0$ by $T$. Positivity is proved in STEP 4. 
	\item{({\it ii})} This part is proved in STEP 3, if one replaces $t_0$ by $T$.
	\item{({\it iii})} This part is proved in STEP 6. 
	\end{list}
 	\rightline{$\square$}
 	\end{proof}

 \section{\bf Conclusions}\label{conc}
 
 In this paper, the derivation of a formula for the effective viscosity formally derived in \cite{RyaHaiBerZieAra11} was made rigorous and an additional term in the asymptotic expansion for the effective viscosity was derived (now up to $O(B^2)$).  This formula revealed the physical mechanisms responsible for the decrease in the effective viscosity confirming the prior formal calculation.  Namely, hydrodynamic interactions, an elongated body, and self-propulsion are required to observe a decrease.  These features are all present in the bacteria {\it Bacillius subtillis} used in the experiments of Aranson et al. \cite{Sokolov2,Sokolov3,Sokolov,SokGolFelAra09,SokAra12}, which motivated this study of the effective viscosity.  In addition, an interesting phenomenon was uncovered: the emergence of self-induced noise where a completely deterministic system governed by interactions resembles a random system for certain regimes of the physical parameters.  The explicit analytical formula for the effective viscosity derived herein showed good qualitative agreement with simulations and experiment.  This paper also establishes the global solvability of solutions to the PDE kinetic equation governing the evolution of the bacterium orientation density.  In order to derive the formula for the effective viscosity, the existence of a steady state was assumed and then computed asymptotically.  Rigorously proving the convergence to a steady state distribution may be the subject of future work.

 \section*{Acknowledgment}
 The authors thank to V.A. Rybalko and I.~S. Aranson for helpful discussions.   The work of LB, MP, and SR were supported by DOE Grant DE-FG-0208ER25862.

 \vspace{0.3 in}
 
%  \noindent {\large \bf Appendices}
 
 \vspace{-0.2 in}

 \bibliographystyle{spmpsci}
 \bibliography{EV_ref}

  \begin{appendix} 
 
 \section{\bf Appendix: Explicit form of integral terms $\text{I}_i$ from \eqref{OB2-int}}\label{KinE-OB2}
 We will need the following technical Lemma: 
 \begin{lemma}
 Assume that $\mathcal{A}$ is a $3\times 3$-matrix that is independent of the orientation vector ${\bf d}$. Then 
 \begin{equation}
 \label{nabla_d_formula}
 \nabla_{\bf d}\cdot [{\bf d}\times ({\bf d}\times \mathcal{A}{\bf d})]=3({\bf d},\mathcal{A}{\bf d})-\text{\rm Tr}\mathcal{A}.
 \end{equation}
 In particular, if  
 	 \begin{equation}\label{particular_A}
 	 \mathcal{A} = \left[\begin{array}{ccc} A & C & 0\\  C & -A & 0 \\ 0 & 0 & 0 \end{array}\right],
 	 \end{equation}
 	 then 
 	  \begin{equation}\label{nabla_d_ACformula}
 	  \nabla_{\bf d}\cdot [{\bf d}\times ({\bf d}\times \mathcal{A}{\bf d})]= A\sin^2\beta\cos(2\alpha) + C\sin^2\beta\sin(2\alpha).
 	  \end{equation}
  
 \end{lemma} 
 \begin{remark}
 	Recall that $\nabla_{\bf d}$ denotes the spherical gradient in orientation ${\bf d}$, and $\tilde{\nabla}_{\bf d}$ denotes the classical gradient in vector ${\bf d}$ (e.g., see \eqref{div2}). 
\end{remark}
\begin{proof}
Using the well-known vector identity $a\times (b \times c)=b(a,c)-c(a,b)$ and the relation \eqref{div2} we obtain 
\begin{eqnarray}
 	&&\nabla_{\bf d} \cdot \left[{\bf d} \times {\bf d} \times\mathcal{A} {\bf d} \right] =  \nabla_{\bf d} \cdot \left[{\bf d} ({\bf d}, \mathcal{A}{\bf d}) - \mathcal{A}{\bf d} \right]\label{app_form_0}\\ &&\hspace{50 pt}= \tilde{\nabla}_{\bf d} \cdot \left[{\bf d}({\bf d}, \mathcal{A}{\bf d}) - \mathcal{A}{\bf d}\right] -\frac{\partial}{\partial |{\bf d} |} \left\{ |{\bf d} |^5(\hat{{\bf d}}, \mathcal{A}\hat{{\bf d}}) - |{\bf d} |^3(\hat{{\bf d}}, \mathcal{A}\hat{{\bf d}})\right\}\biggr|_{|{\bf d}| =1}.\nonumber
 	\end{eqnarray}
Here $\hat{\bf d}={\bf d}/|{\bf d}|$.  The orientation ${\bf d}$ is a unit vector, but in order to relate the classical and the spherical divergence we need to calculate the derivative in $|{\bf d}|$ at $|{\bf d}|=1$; thus, consider ${\bf d}$ different from unit magnitude.  Also, note that $\hat{\bf d}$ does not depend on $|{\bf d}|$.

One can easily verify that 
 \begin{align}
  \tilde{\nabla}_{\bf d} \cdot \left[{\bf d}({\bf d}, \mathcal{A}{\bf d})-\mathcal{A}{\bf d}\right] &= 3({\bf d}, \mathcal{A}{\bf d}) + {\bf d} \cdot \tilde{\nabla} ({\bf d}, \mathcal{A}{\bf d}) - \text{Tr}(\mathcal{A})\nonumber\\
 &=3({\bf d}, \mathcal{A}{\bf d}) + {\bf d} \cdot  2\mathcal{A}{\bf d} - \text{Tr}(\mathcal{A})\nonumber\\
 &=5({\bf d}, \mathcal{A}{\bf d})- \text{Tr}(\mathcal{A}).\label{app_form_1}
 \end{align}
 and 
 \begin{align}
   -\frac{\partial}{\partial |{\bf d} |} \left\{ |{\bf d} |^5(\hat{{\bf d}}, \mathcal{A}\hat{{\bf d}}) - |{\bf d} |^3(\hat{{\bf d}}, \mathcal{A}\hat{{\bf d}})\right\}\biggr|_{|{\bf d}| =1} = -2({\bf d}, \mathcal{A}{\bf d}).\label{app_form_2}
 \end{align}
 Substituting \eqref{app_form_1} and \eqref{app_form_2} into \eqref{app_form_0} we obtain \eqref{nabla_d_formula}. The formula \eqref{nabla_d_ACformula} follows directly from substituting \eqref{particular_A} into \eqref{nabla_d_formula}.

	\rightline{$\square$} \end{proof}

Next we evaluate integral terms $\text{I}_i$ introduced in Subsection \ref{contr_at_B2}.  

\medskip 

 \noindent{\it The integral term $\text{\rm I}_1$.} This integral is defined by 
 \begin{equation*}
 \text{I}_1:=\frac{1}{4\pi N|V_L|} \nabla_{{\bf d}} \cdot \int_{\mathcal{S}^2} \langle F[{\bf E}]F[P_{\x}]^2 \rangle_{{\bf k}} P_{{\bf d}}^{(1)}({\bf d}')dS_{{\bf d}'}
 \end{equation*}
  and due to \eqref{fourier_du} and \eqref{FourierE} can be written as:
\begin{align*}
 	\text{I}_1 =\frac{-1}{8\eta\pi N|V_L|} \nabla_{{\bf d}} \cdot \left[{\bf d} \times {\bf d} \times \left\{\int  \mathcal{M}F[P_{\bf x}]^2 d{\bf k} \right\} {\bf d} \right].
 	\end{align*}
 	Here $\hat{\bf k}:={\bf k}/|{\bf k}|$, and the $3\times 3$ matrix $\mathcal{M}$ is defined by 
 	\begin{equation*}
 	\mathcal{M}:=  \mathcal{F} \hat{\bf k}\hat{\bf k}^* - 2\hat{\bf k}\hat{\bf k}^* \mathcal{F}\hat{\bf k}\hat{\bf k}^* + \hat{\bf k}\hat{\bf k}^* \mathcal{F}, 
 	\end{equation*} where  
  \begin{equation}\label{matrix_F}
 \mathcal{F} := \left[\int_{\mathcal{S}^2} \tilde{\Sigma} P_{{\bf d}}^{(1)}({\bf d}') dS_{{\bf d}'}\right] = -\frac{3U_0}{8\pi}\left[\begin{array}{ccc} \frac{8\pi}{15} & 0 & 0\\
 0 & -\frac{8\pi}{15} & 0\\
 0 & 0 & 0\end{array}\right]= -\frac{U_0}{5}\left[\begin{array}{ccc} 1 & 0 & 0\\
 0 & -1 & 0\\
 0 & 0 & 0\end{array}\right].
 \end{equation}

 Substituting $\mathcal{F}$ into the expression for $\mathcal{M}$ one finds that $\mathcal{M}$ equals to 
  {\small \begin{equation*}
  	\left[\begin{array}{ccc} 2k_1^2 - 2k_1^4 + 2k_1^2k_2^2 & -2k_1^3k_2 + 2k_1k_2^3 & k_1k_3 -2k_1^3k_3 + 2k_1k_2^2k_3\\
  	-2k_1^3k_2 + 2k_1k_2^3 & -2k_2^2 - 2k_1^2k_2^2 + 2k_2^4 & -k_2k_3 - 2k_1^2k_2k_3 + 2k_2^3k_3\\
  	k_1k_3 -2k_1^3k_3 + 2k_1k_2^2k_3 & -k_2k_3 - 2k_1^2k_2k_3 + 2k_2^3k_3 & -2k_1^2k_3^2 + 2k_2^2k_3^2\end{array}\right],
  	\end{equation*}}
  where $k_1,k_2,k_3$ are components of $\hat{\bf k}$.
  
 Next we use the condition (C3) from Subsection \ref{ass_a3} written in terms of the representation formula \eqref{ass_a3_formula}  to obtain  
 \begin{equation}\nonumber
  \int \mathcal{M}(F[P_{\x}])^2 d{\bf k}=\int \mathcal{M}|_{k_3=0}\hat{P}_{12}(|{\bf k}|k_1,|{\bf k}|k_2)|{\bf k}|^2 d|{\bf k}|d\theta,
  \end{equation}
 where 
 {\footnotesize
 	\begin{align*}
 	\mathcal{M}|_{k_3=0} &= \left[\begin{array}{ccc} 2k_1^2[1 - k_1^2 + k_2^2] & -2k_1k_2[k_1^2 - k_2^2] & 0 \\
 	-2k_1k_2[k_1^2 - k_2^2] & -2k_2^2[1 + k_1^2 - k_2^2] & 0\\ 0 & 0 & 0 \end{array}\right]\\
 	&  = \left[\begin{array}{ccc} 4k_1^2k_2^2 & -2k_1k_2[k_1^2 - k_2^2] & 0 \\
 	-2k_1k_2[k_1^2 - k_2^2] & -4k_1^2k_2^2 & 0\\ 0 & 0 & 0 \end{array}\right]  = \left[\begin{array}{ccc} \sin^2(2\theta) & -\frac{1}{2}\sin(4\theta) & 0 \\
 	-\frac{1}{2}\sin(4\theta) & -\sin^2(2\theta) & 0\\ 0 & 0 & 0 \end{array}\right]. 
 	\end{align*}}
 Here  variables $k_1$ and $k_2$ are expressed in polar coordinates $k_1 = \cos(\theta)$ and $k_2 = \sin(\theta)$.

Note that the matrix in the equality above is of the form \eqref{particular_A} and, thus, we may apply \eqref{nabla_d_ACformula}: 
\begin{equation*}
\text{I}_1=\frac{ U_0}{40\pi \eta_0 N|V_L|}\left(A\sin^2\beta\cos(2\alpha) + C\sin^2\beta\sin(2\alpha)\right).
\end{equation*}
 This is the desired formula for $\text{I}_1$. 
 
 \medskip 
 
 \noindent{\it The integral term $\text{\rm I}_2$.} We need to obtain that $\text{I}_2=0$. This holds provided that
 \begin{equation}\label{i2_vanish}
 \int\limits_{\mathcal{S}^2}\langle F[\boldsymbol{\omega}](F[P_{\bf x}])^2\rangle_{\bf k} dS_{{\bf d}'}=0.
 \end{equation}

The integral in the RHS of \eqref{i2_vanish} by using inverse Fourier transform can be written as 
\begin{equation*}
-\frac{1}{2}{\bf d} \times \int\int P_{\x}(\x')P_{\x}(\x)\nabla_{\x}\times\left\{\int_{\mathcal{S}^2}{\bf u}(\x-\x',{\bf d}')dS_{{\bf d}'}\right\}d{\bf x}d{\bf x}'.
\end{equation*} 
The integral in curly braces is zero due to 
\begin{equation}
\label{int_u_is_zero}
\int_{\mathcal{S}^2} {\bf u}({\bf x},{\bf d}') d {\bf d}'=\left[\int_{\mathcal{S}^2}U_0({\bf d}'({\bf d}')^*-I/3)dS_{{\bf d}'}\right]:\nabla_{\bf x}\mathcal{G}=0.
\end{equation}
Thus, \eqref{i2_vanish} holds and  $\text{I}_2=0$.

 \medskip 
 
 \noindent{\it The integral term $\text{\rm I}_3$.} To prove that $\text{I}_3=0$ we can use the same arguments as for $\text{I}_2$. Indeed, $\text{I}_3$ vanishes provided that 
   \begin{equation}\nonumber%\label{i3_vanish}
   \int\limits_{\mathcal{S}^2}\langle F[{\bf E}](F[P_{\bf x}])^2\rangle_{\bf k} dS_{{\bf d}'}=0.
   \end{equation}
One can easily verify this equality by using the inverse Fourier transform and the identity \eqref{int_u_is_zero}.  

 \medskip 
 
 \noindent{\it The integral term $\text{\rm I}_4$.} This integral can be written as 
 \begin{align*}
 \frac{\nabla_{{\bf d}}[P^{(1)}_{{\bf d}}({\bf d})]}{N|V_L|} \int_{\mathcal{S}^2} P^{(1)}_{{\bf d}}({\bf d}') \langle F[{\boldsymbol \omega}](F[P_{\bf x}])^2\rangle_{{\bf k}} dS_{{\bf d}'}.
 \end{align*}
 According to \eqref{Fourierc}, the formula for $F[\boldsymbol{\omega}]$ is the following
 \begin{align*}
 F[{\boldsymbol \omega}] = -\frac{\bf d}{2} \times [-i{\bf k} \times \tilde{\bf u}] = -\frac{{\bf d}}{2\eta} \times \left[\hat{{\bf k}} \times \left(I - \hat{{\bf k}}\hat{{\bf k}}^*\right)\tilde{\Sigma}\hat{{\bf k}}\right].
 \end{align*}
 Recall that $\hat{{\bf k}} = {\bf k} / |{\bf k} |=(k_1,k_2,k_3)$. Using \eqref{matrix_F} we obtain
 \begin{eqnarray*}
 	M:=\int_{\mathcal{S}^2} F[\boldsymbol{\omega}]P_{{\bf d}}^{(1)}({\bf d}') dS_{{\bf d}'}&=&-\frac{1}{2\eta}{\bf d}\times\left[\hat{{\bf k}} \times \left(I - \hat{{\bf k}}\hat{{\bf k}}^*\right)\mathcal{F}\hat{{\bf k}}\right]\\
 	&=&\frac{U_0}{10\eta}{\bf d} \times \left[\begin{array}{c} k_2k_3 \\ k_1k_3\\ -2k_1k_2\end{array}\right].
 \end{eqnarray*}

   In the same manner as we analyzed $\text{I}_1$, we use the condition (C3) from Subsection~\ref{ass_a3} written in terms of the representation formula \eqref{ass_a3_formula}, the form of orientation ${\bf d}$ in spherical angles \eqref{d_spherical}, and polar angle $\theta$ for $k_1=\cos\theta$ and $k_2=\sin\theta$:
   \begin{equation}\label{i4_f_formula}
   \text{I}_4=\frac{\nabla_{{\bf d}}[P^{(1)}_{{\bf d}}({\bf d})]}{N|V_L|}\cdot\int M|_{k_3=0}\hat{P}^2_{12}(|{\bf k}|k_1,|{\bf k}|k_2)|{\bf k}|^2 d|{\bf k}|d\theta,
   \end{equation}
   where 
   \begin{equation*}
   M|_{k_3=0}=\frac{U_0}{10\eta}\sin (2\theta)\left[\begin{array}{c} -\sin\alpha \sin\beta \\ \cos\alpha\sin\beta \\0\end{array}\right].
 \end{equation*}
Next we find $\nabla_{{\bf d}}[P^{(1)}_{{\bf d}}({\bf d})]$. Using \eqref{pd_at_order_1} and the definition of the spherical gradient $\nabla_{\bf d}$:
 \begin{equation*}
 \nabla_{{\bf d}}P = \left[\begin{array}{c} -\frac{\sin(\alpha)}{\sin(\beta)}\partial_{\alpha}P + \cos(\alpha)\cos(\beta)\partial_{\beta}P\\
 \frac{\cos(\alpha)}{\sin(\beta)}\partial_{\alpha}P + \sin(\alpha)\cos(\beta)\partial_{\beta}P\\
 -\sin(\beta)\partial_{\beta}P\end{array}\right],
 \end{equation*}
 we obtain that
 \begin{align*}
 \nabla_{{\bf d}}P_{{\bf d}}^{(1)}({\bf d}) &= \frac{-3}{4\pi}\left[\begin{array}{c} \sin\alpha\sin(2\alpha)\sin\beta + \cos\alpha\cos(2\alpha)\sin\beta\cos^2\beta\\
 -\cos\alpha\sin(2\alpha)\sin\beta + \sin\alpha\cos(2\alpha)\sin\beta\cos^2\beta\\
 -\sin^2\beta\cos\beta\cos(2\alpha)\end{array}\right].
 \end{align*}
 Substituting this equality into \eqref{i4_f_formula} one obtains the desired formula for $\text{I}_4$: 
 \begin{equation*}
 \text{I}_4=\frac{3U_0}{10\pi\eta_0 N|V_L|}D\sin(2\alpha)\sin^2\beta.
 \end{equation*}
 This concludes the evaluation of integral terms $\text{I}_i$ for $i = 1,...,4$. 
 
 \section{\bf Appendix: Justification of the representation formula \eqref{ass_a3_formula}}\label{pxapp}
 
 Consider the spatial distribution $P_{\x}(x,y,z)=c_L\chi(z)P_{12}(x,y)$, where 
 \begin{equation}
 \chi(z)=\left\{\begin{array}{lr}1,& z\in (-L,L),\\
 0,& z\notin (-L,L), \end{array}\right.
 \end{equation}
and we choose $c_L=4/\sqrt{\pi L}$. The distribution $P_{\x}$ satisfies the condition (C3), i.e., its support does not depend on $z$.   

Our main goal of this subsection is to obtain a representation for the Fourier transform of $P_{\x}$: 
\begin{equation}\label{px_fourier}
F[P_{\bf x}]=\int_{-L}^{L}\chi( z) e^{i k_3 z}dk_3 \hat{P}_{12} (k_1,k_2)=-\frac{2}{\sqrt{\pi L}}\frac {\sin (k_3 L)}{k_3}\hat{P}_{12} (k_1,k_2). 
\end{equation}

For an arbitrary continuous function $\phi$ the following convergence holds:
\begin{equation}\label{px_converge}
\frac{1}{\pi L}\int\limits_{-L}^{L} \left(\frac{\sin (k_3 L)}{k_3}\right)^2\phi (k_3)dk_3\to \phi(0). 
\end{equation}

From \eqref{px_fourier} and \eqref{px_converge} it follows that for large $L$ we have 
\begin{equation}
(F[P_{\bf x}])^2\approx \delta (k_3) \hat{P}^2_{12}(k_1,k_2), 
\end{equation}
which justifies \eqref{ass_a3_formula}. Note that due to our choice of $c_L$ it follows from $\int_{V_L} P_{\x}d\x=N$ and $N\sim L^3$  that $P_{12}\sim \sqrt{L}$.  

It is also interesting to calculate the coefficient $A$ defined by \eqref{coefs_ACD} for the spatial distribution $P_{\x}(\x)=\frac{1}{\rho}\chi(x)\chi(y)\chi(z)$ 
which is uniform in $V_L$. Then
\begin{equation}\label{hatP_asymp}
\hat{P}_{12}^2=\sqrt{L} \left(\frac{\sin (k_1 L)}{k_1}\right)^2\left(\frac{\sin (k_2 L)}{k_2}\right)^2\sim  \delta(k_1,k_2) L^{5/2}.
\end{equation}

 In this case the coefficient $A$ is of the order $L^{5/2}$. It is responsible for the decrease in viscosity. Namely, for fixed number density $\rho=N/L^3$, the Bretherton constant $B$, the dipole moment $U_0$ and the strength of the background flow $\gamma$ it follows from \eqref{eqn:evformula} that $\eta^{\text{int}}\sim A/L^6$. Then due to \eqref{hatP_asymp}
\begin{equation}
\eta^{\text{int}}\sim \frac{1}{L^{7/2}}\to 0 \text{ as } L\to\infty.   
\end{equation} 
Therefore, $\tilde{A}=A L^{-5/2}$ can serve as a measure of the deviation of the spatial density $P_{\bf x}(\x)$ from uniform.

\end{appendix}
\end{document}